\documentclass[12pt,final,reqno]{amsart}
\usepackage{ntung}
\usepackage{cleveref}
\usepackage{thm-restate}

\usepackage{ntung_commands}

\algrenewcommand\algorithmicrequire{\textbf{Input:}}
\algrenewcommand\algorithmicensure{\textbf{Output:}}

\newcommand{\C}{\mathcal{C}}

\DeclareMathOperator{\lmax}{\lambda_{\text{max}}}

\begin{document}

\title{New Sidorenko-type inequalities in tournaments}

\author{Xiaoyu He}
\author{Nitya Mani}
\author{Jiaxi Nie}
\author{Nathan Tung}
\author{Fan Wei}

\address{Georgia Institute of Technology, Atlanta, GA 30332, USA}
\email{\{xhe399,jnie47\}@gatech.edu}

\address{Massachusetts Institute of Technology, Cambridge, MA 02139, USA}
\email{nmani@mit.edu}

\address{Stanford University, CA 94305, USA}
\email{ntung@stanford.edu}

\address{Duke University, Durham, NC 27708, USA}
\email{fan.wei@duke.edu}

\begin{abstract}
As a directed analog of Sidorenko's conjecture in extremal graph theory, Fox, Himwich, Zhou, and the second author defined an oriented graph $H$ to be \textit{tournament Sidorenko} (\textit{anti-Sidorenko}) if the random tournament asymptotically minimizes (maximizes) the number of copies of $H$ among all tournaments. We prove new inequalities of this form for oriented trees and cycles, considering both local and global notions of being Sidorenko.

We make progress on a conjecture of the aforementioned authors that every tree has an anti-Sidorenko direction, and give a characterization of short paths. For long paths we show that orientations are split symmetrically between being locally Sidorenko and anti-Sidorenko, yet almost all orientations are not globally Sidorenko. Finally, we give algorithms characterizing the local Sidorenko status of paths and cycles when the number of vertices is not divisible by four.
\end{abstract}
\maketitle

\section{Introduction}\label{sec:intro}
For graphs $H$ and $G$, a homomorphism from $H$ to $G$ is a function from $V(H)$ (the vertex set of $H$) to $V(G)$ that sends edges to edges. Let $\hom(H,G)$ be the number of homomorphisms from $H$ to $G$, and let
$$
t_H(G) = \frac{\hom(H,G)}{v(G)^{v(H)}}
$$
denote the homomorphism density of $H$ in $G$. Sidorenko's conjecture, posed independently by Erd\H{o}s-Simonovits~\cite{SIM82} and Sidorenko~\cite{SID93}, states that for any bipartite graph $H$ and any graph $G$
\begin{equation}\label{eq:origsiddef}
    t_H(G) \ge t_{K_2}(G)^{e(H)},
\end{equation}
where $K_2$ is the single edge graph. Equivalently, among all graphs $G$ with the same vertex count $n$ and edge density $p$, the Erd\H{o}s-R\'enyi random graph $G(n,p)$ asymptotically minimizes $t_H(G)$ (in expectation). This surprising statement, that a random graph is essentially the minimizer for a simple extremal problem, has been the subject of intense study over the last three decades, and many cases have been proved \cites{SID93,CON10,LI17,KIM16,CON18,SZE}. 

The analogous question for directed graphs was proposed by Fox, Himwich, Zhou, and the second author \cite{genFHMZ22}: under what conditions does the random directed graph $G$ minimize the number of copies of a directed graph $D$? Roughly speaking, they showed that the directed Sidorenko's conjecture is equivalent to the (slightly stronger) asymmetric undirected version, and no new phenomena arise in this context~\cite{genFHMZ22}. In follow-up work \cite{fox2024variationssidorenkosconjecturetournaments}, they discovered that in a suitably restricted version of this problem, fixing $D$ to be directed and $G$ to be a tournament, surprising new phenomena do arise, including the existence of \textit{anti-Sidorenko} inequalities where the random tournament $G$ is the asymptotic maximizer for $t_D(G)$ among all tournaments with the same vertex count. A special case of this problem was previously solved by Zhao and Zhou~\cite{ZHAO19}, who characterized the family of directed graphs $D$ that are impartial, meaning $t_D(G)$ is invariant to the choice of tournament $G$ (for example, a single directed edge). More recently Sah, Sawhney, Zhao \cite{SSZ20} and Grzesik, Král', Lovász, Volec \cite{cycles} have studied the maximum density of directed paths and cycles into tournaments, where all edges of the path or cycle are oriented the same way (left to right or clockwise). This paper investigates these tournament Sidorenko-style inequalities further, focusing on other orientations of trees and cycles.

There are several equivalent formulations of what it means for a directed graph to be tournament Sidorenko or anti-Sidorenko. One of the simplest is that a directed graph $D$ is TS (\textit{tournament Sidorenko}) if
\begin{equation}\label{eq:siddef}
    t_D(G) \ge (1-o(1))t_{K_2}(G)^{e(D)}
\end{equation}
holds for any tournament $G$, and $D$ is TAS (\textit{tournament anti-Sidorenko}) if
\begin{equation}\label{eq:antisiddef}
    t_D(G) \le t_{K_2}(G)^{e(D)}
\end{equation}
holds for any tournament $G$. The asymmetry is due to the assumed lack of self-loops in $G$. Note that if $G$ is an $n$-vertex tournament, then $t_{K_2}(G)=\frac{\binom{n}{2}}{n^2}=\frac{1}{2}-\frac{1}{2n}$.

The other formulations come from either replacing $t_D(G)$ with another normalization or requiring the inequalities to hold for a larger class of objects. Define a \textit{weighted tournament} on $n$ vertices to be an adjacency matrix $A: [n]^2 \to [0,1]$, and a \textit{tournamenton} to be a function $A: [0,1]^2 \to [0,1]$. In both cases we require $A(x,y) + A(y,x) = 1$ for all $x,y$, which imposes the tournament-like condition. For a weighted tournament $A$ and directed graph $D$, define the weighted homomorphism count of $D$ into $A$ to be
$$
h_D(A) \coloneqq \sum_{\phi: V(D) \to [n]} \prod_{(x,y) \in E(D)} A(\phi(x),\phi(y))
$$
and extend $t_D$ to weighted tournaments or tournamentons $A$ by
\begin{equation}\label{eq:hom-density}
t_D(A)=\int \prod_{(x_1,x_2)\in E(D)}A(x_1,x_2)\prod_{x\in V(D)}dx,
\end{equation}
where the integral is expectation over uniform $x_1,\dots,x_{V(D)}$ in $[0,1]$ or $[n]$. Then it is equivalent in the definition of TS/TAS to require that the inequalities hold not just for all tournaments but for all weighted tournaments or tournamentons, and in the case of weighted tournaments also to replace \eqref{eq:siddef} or \eqref{eq:antisiddef} with
$$
h_D(A) \ge \frac{n^{v(D)}}{2^{e(D)}} \text{ or } h_D(A) \le \frac{n^{v(D)}}{2^{e(D)}}
$$
respectively. This follows from the fact that tournamentons can be approximated by weighted tournaments, which in turn can be approximated by unweighted tournaments (see \cite{cycles} for more on tournamentons).

The following conjecture of Fox, Himwich, Zhou, and the second author appeared in \cite{fox2024variationssidorenkosconjecturetournaments}.
\begin{conjecture}[TAS trees]\label{conj:originaltree}
    For every undirected tree, there exists some orientation that has the tournament anti-Sidorenko property.
\end{conjecture}
In the same paper they proved that it holds for any tree with exactly one vertex of even degree. Our first main results reduce this conjecture by proving it for new families of trees. The first such result actually applies to general graphs, and the problem of finding TAS directions of arbitrary undirected graphs.

\begin{definition}\label{def:isopair}
    Let $G$ be an undirected graph. Call disjoint subsets $H_1,H_2 \subset V(G)$ with $e(H_1,H_2) = \emptyset$ an \textit{isomorphic pair} if there exists an isomorphism of induced subgraphs $\phi: G[H_1] \to G[H_2]$ and vertices $w \in H_1, v \in V(G)$ such that the cut $e(H_1 \cup H_2,V(G) \setminus (H_1 \cup H_2)) = \set{\set{v,w},\set{v,\phi(w)}}$.
\end{definition}
In the case $G$ is a tree and one wants to think of it as rooted, essentially $H_1,H_2$ are isomorphic subtrees sharing a parent. For example two leaves sharing a parent form an isomorphic pair, as do the vertices $\set{1,2,\dots,k-1}, \set{k+1,k+2,\dots,2k}$ in a path with vertex set $[2k]$.
\begin{restatable}[Isomorphic pair pruning]{theorem}{isopair}\label{thm:isopairtas}
    Let $G$ be an undirected graph with isomorphic pair $H_1,H_2$. If $G[V(G) \setminus (H_1 \cup H_2)]$ and $G[H_1]$ have TAS orientations, then so does $G$.
\end{restatable}

In the case of trees we can say more. Call a tree a \textit{caterpillar} if removing all leaves results in a path.

\begin{theorem}[Caterpillar]\label{thm:caterpillar}
Every caterpillar has a TAS orientation.
\end{theorem}

Theorem \ref{thm:caterpillar} along with inductively applying Theorem \ref{thm:isopairtas} within each isomorphic pair reduces Conjecture \ref{conj:originaltree} to the following.

\begin{conjecture}
    For every undirected tree that is not a caterpillar with at least $2$ vertices of even degree and no isomorphic pair, there exists some orientation that has the tournament anti-Sidorenko property.
\end{conjecture}

A small tree for which Conjecture \ref{conj:originaltree} remains open is the $2-3-4$ tree. 
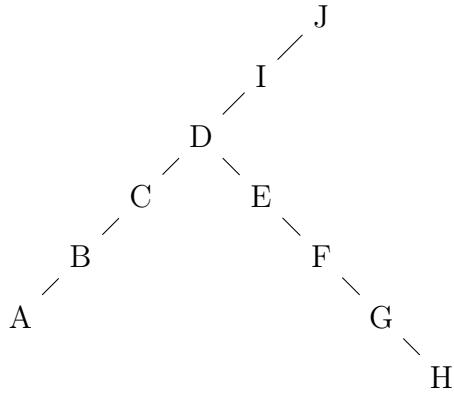
\begin{center}
\begin{tikzpicture}[scale=0.8]

  \node (A) at (0,0) {B};
  \node (B) at (1,1) {C};
  \node (C) at (2,2) {D};
  \node (G) at (3,3) {I};
  \node (D) at (3,1) {E};
  \node (E) at (4,0) {F};
  \node (F) at (5,-1) {G};

  \node (H) at (-1,-1) {A};
  \node (I) at (6,-2) {H};
  \node (J) at (4,4) {J};

  \draw (H) -- (A);
  \draw (A) -- (B);
  \draw (B) -- (C);
  \draw (C) -- (G);
  \draw (G) -- (J);

  \draw (I) -- (F);
  \draw (F) -- (E);
  \draw (E) -- (D);
  \draw (D) -- (C);
\end{tikzpicture}
\end{center}

\begin{problem}
    Find a TAS orientation of the $2-3-4$ tree.
\end{problem}

Our other main results concern the behavior of long oriented paths with respect to both local and global notions of being Sidorenko and anti-Sidorenko. We call a directed graph $D$ LTS (\textit{locally tournament Sidorenko}) if \eqref{eq:siddef} holds for all tournaments $G$ that are sufficiently quasirandom, and LTAS (\textit{locally tournament anti-Sidorenko}) if \eqref{eq:antisiddef} holds for all sufficiently quasirandom $G$. The uniformity norm we will use to measure quasirandomness will be the spectral radius/cut norm; more on this and relations between uniformity norms in Section~\ref{sec:local}. In the undirected setting, the local version of Sidorenko is defined similarly: \eqref{eq:origsiddef} holds for all graphs $G$ close to quasirandom. The local version of Sidorenko's conjecture, that any bipartite graph is locally Sidorenko, was proved in \cite{loclov}. Note that, unlike for globally Sidorenko, being bipartite is not necessary for a graph to be locally Sidorenko. It was subsequently shown in \cite{locfox} that a graph is locally Sidorenko if and only if it has even girth or is a forest. Since locally Sidorenko graphs are completely understood in the undirected setting despite Sidorenko's conjecture remaining wide open, it is natural to hope that local versions of TS/TAS may be more tractable than the global notions.

We provide conditions under which oriented paths and cycles are LTS, LTAS, or neither and are able to achieve a full characterization when the number of vertices is not divisible by four. These conditions are easy to check, and lead to feasible algorithms for characterization. With regards to the typical orientation of long paths, we establish the following.
\begin{proposition}[Local Sidorenko symmetry]\label{prop:localwalk}
    For odd $n$ exactly $\frac{1}{2}$ of all oriented paths on $n+1$ edges are locally TS and exactly $\frac{1}{2}$ are locally TAS. For even $n$ the fraction of oriented paths on $n+1$ edges which are neither locally TS or locally TAS is at most ${n \choose n/2}2^{-n}$ and the number that are locally TS/TAS are each at least $\frac{1}{2}(1-{n \choose n/2}2^{-n})$.
\end{proposition}

Turning to the more difficult global notions of Sidorenko, we characterize the TS/TAS status of all oriented paths of length at most $5$. For long paths, as in the local setting, there are few TS orientations.
\begin{theorem}[TS rarity]\label{thm:rarity}
    The fraction of oriented paths on $n$ edges that are TS is $o_n(1)$.
\end{theorem}
In combination with Proposition \ref{prop:localwalk}, this implies that at least $\frac 1 2 - o_n(1)$ of all oriented paths are neither TS nor TAS. The question remains whether almost all are neither or whether many turn out to actually be TAS. Despite the implied asymmetry, we conjecture the latter

\begin{conjecture}[TAS ubiquity]\label{conj:ubiquity}
    A $(\frac 1 2 - o_n(1))$ fraction of all oriented paths of length $n$ are TAS.
\end{conjecture}

A variety of new ideas are used to prove these statements, including spectral methods, entropy, and a connection to stochastic linear recurrences that has been studied extensively in probability theory \cites{visfib,goldsheid2025exponentialgrowthrandominfinite}.

\subsection{Structure of the paper}

Section \ref{sec:polyexpansion} covers the polynomial expansion and spectral methods. We use these to give new proofs of some known results and characterize short paths. \Cref{sec:local} provides characterizations and classifying algorithms in the local setting for oriented paths and cycles. Section \ref{sec:typical} proves \Cref{prop:localwalk} and Theorem \ref{thm:rarity}. Section \ref{sec:treecon} proves Theorems \ref{thm:isopairtas} and \ref{thm:caterpillar} and also takes a different approach to Conjecture \ref{conj:originaltree} by asking which undirected graphs have a TAS orientation. Proposition \ref{prop:sparse} constructs a sparse graph with no TAS orientation.

\subsection{Acknowledgements}

The authors thank David Conlon and Jacob Fox for helpful discussions, as well as Benjamin McKenna for pointing us to the literature on stochastic linear recurrences. Fan Wei was partially supported by NSF grant DMS-2401414.

\section{Polynomial expansion and spectral methods}\label{sec:polyexpansion}

We begin with a spectral argument particularly suited for counting paths and cycles. Spectral methods have previously been used to count directed cycles in tournaments \cite{cycles}. All matrices are $n\times n$ where $n$ is the number of vertices of the tournament $T$. Let $A$ be the adjacency matrix of a weighted tournament, with $1/2$ along the diagonal to indicate weighted self-loops. Let $J$ be the all-ones matrix, and observe that one may decompose $A = B + \frac 12 J$ where $B$ is skew-symmetric. The spectral theorem then implies that $B$ has an orthonormal basis of eigenvectors with eigenvalues $i\lambda$ all pure imaginary. 
$B$ can be viewed as the structured perturbation from the quasirandom $\frac 1 2 J$. We illustrate the general method with a previously established result.

\begin{proposition}\label{prop:p2ts}
    The path on $3$ vertices $\rightarrow \leftarrow$ is TS.
\end{proposition}
\begin{proof}
    Let $A$ be the adjacency matrix of a tournament $T$ and $P = \rightarrow \leftarrow$. Expanding about the quasirandom part we obtain
    \begin{align}\label{eq:expansion}
    h_P(A) = 1^\top  AA^\top  1 &= 1^\top  \bigp{\frac 12 J + B}\bigp{\frac 12 J - B} 1 \\
    &= \frac 1 4 1^\top  J^2 1 - \frac 12 1^\top  JB 1 + \frac 12 1^\top  BJ 1 - 1^\top  B^2 1 = \frac{n^3}{4} - 1^\top  B^2 1. \nonumber
    \end{align}
    Since $B^2$ is negative semidefinite the final count is at least $n^3/4$, as desired.
\end{proof}

The method is able to handle more interesting oriented paths that previous methods cannot. For example, the path $\rightarrow\rightarrow\leftarrow\leftarrow$ can be proved to be TAS (Proposition \ref{prop:signflip}). Using the polynomial expansion method we can characterize all paths of length at most 5 (Table \ref{tab:paths}). Proofs are omitted, but only require inequality (iii) of Lemma \ref{lemma:estimating X}, skew symmetry of $B$, and that $B$ has a spectral radius of at most $n/2$.

\begin{table}[h]
\centering
\begin{minipage}{.45\linewidth}
\centering
\begin{tabular}{c|c}
$\rightarrow$         & Impartial    \\
$\rightarrow\leftarrow$         & TS      \\
    $\rightarrow\rightarrow$         & TAS       \\
  $\rightarrow\rightarrow\rightarrow$         & TAS       \\
  $\leftarrow\rightarrow\rightarrow$         & Impartial \\
  $\leftarrow\rightarrow\leftarrow$          & TS        \\
  $\rightarrow\rightarrow\rightarrow\rightarrow$ & TAS   \\
  $\rightarrow\rightarrow\rightarrow\leftarrow$   & TAS   \\
  $\rightarrow\rightarrow\leftarrow\rightarrow$   & TS    \\
  $\rightarrow\rightarrow\leftarrow\leftarrow$    & TAS   \\
  $\rightarrow\leftarrow\rightarrow\leftarrow$    & TS 
\end{tabular}
\end{minipage}%
\hspace{0.06\linewidth}
\begin{minipage}{.45\linewidth}
\centering
\begin{tabular}{c|c}
  $\rightarrow\leftarrow\leftarrow\rightarrow$    & TS    \\
  $\rightarrow\rightarrow\rightarrow\rightarrow\rightarrow$ & TAS \\
  $\rightarrow\rightarrow\rightarrow\rightarrow\leftarrow$   & TAS \\
  $\rightarrow\rightarrow\rightarrow\leftarrow\rightarrow$   & TAS \\
  $\rightarrow\rightarrow\leftarrow\rightarrow\rightarrow$   & TAS \\
  $\rightarrow\rightarrow\rightarrow\leftarrow\leftarrow$    & TAS \\
  $\rightarrow\rightarrow\leftarrow\rightarrow\leftarrow$    & TS  \\
  $\rightarrow\leftarrow\rightarrow\rightarrow\leftarrow$    & TS  \\
  $\leftarrow\rightarrow\rightarrow\rightarrow\leftarrow$    & TAS \\
  $\rightarrow\rightarrow\leftarrow\leftarrow\rightarrow$    & TS  \\
  $\rightarrow\leftarrow\rightarrow\leftarrow\rightarrow$    & TS  
\end{tabular}
\end{minipage}
\caption{Sidorenko classification of all paths of length at most $5$.}
\label{tab:paths}
\end{table}

We also find the shortest path that is neither TAS nor TS.
\begin{proposition}
    The path with 6 edges $\rightarrow \leftarrow \rightarrow \rightarrow \rightarrow \leftarrow$ is neither TAS nor TS.
\end{proposition}
\begin{proof}
    Consider the transitively oriented triangle $\cT$ with edge weights $1$ and self-loops of weight $1/2$. Then the weighted homomorphism count of our oriented path $P$ into $T$ can be calculated as
    $$
    h_P(\cT) = 1^\top A A^\top A^3 A^\top 1, \quad A = \begin{pmatrix}
        1/2 & 1 & 1\\
        0 & 1/2 & 1\\
        0 & 0 & 1/2
    \end{pmatrix}.
    $$
    It can be explicitly computed that $h_P(\cT) \approx 36.05 > 34.171875 = \frac{3^7}{2^6}$ so $P$ is not TAS. Note that if one wanted a sequence of unweighted tournaments providing a counterexample to our first formulation of TAS it suffices to simply take $\cT_n$ on $n$ vertices attained by blowing up the transitively oriented triangle with quasirandom tournaments on $n/3$ vertices. In general any weighted construction with self-loops of weight $1/2$ can be turned into a sequence with quasirandom blowups.

    For the other direction let $\cT$ now be the cyclically oriented triangle with self-loops weighted $1/2$ but perturbed slightly so that one edge has weight $0.01$ in the reverse direction and $0.99$ in the cyclic direction, with the other two edges still being weight $1$. That is
    $$
    h_P(\cT) = 1^\top A A^\top A^3 A^\top 1, \quad A = \begin{pmatrix}
        1/2 & 0.99 & 0\\
        0.01 & 1/2 & 1\\
        1 & 0 & 1/2
    \end{pmatrix}.
    $$
    Then $h_P(\cT) \approx 34.17178 < 34.171875 = \frac{3^7}{2^6}$ so $P$ is not TS.
\end{proof}

More generally, we make the following conjecture.
\begin{conjecture}[Direction flip]\label{conjecture:one flip}
    For all $k \ge 3$, an oriented path with $k$ edges and one direction flip ($\rightarrow\dots \rightarrow \leftarrow\dots \leftarrow$) is TAS.
\end{conjecture}

Note that it cannot hold for $k=2$ by Proposition \ref{prop:p2ts}. The polynomial expansion approach can be used to confirm this conjecture for paths of length up to $7$. The proof is deferred to Appendix \ref{sec:signflipapp}.

\begin{proposition}\label{prop:signflip}
For all $3\le k\le 7$, an oriented path with $k$ edges and one direction flip is TAS.
\end{proposition}

\section{Locally tournament Sidorenko and anti-Sidorenko}\label{sec:local}

The polynomial expansion approach is especially well suited to studying local Sidorenko properties. For the ``local'' part to make sense one needs to equip the space of graphs with a distance, and in \cite{loclov} this is taken to be the cut distance. The \textit{cut distance} between $A,C \in \text{Mat}_n([0,1])$ is $\norm{A - C}_\Box$ where the cut norm is defined for $B \in \text{Mat}_n([-1,1])$ as
$$
\norm{B}_\Box \coloneqq \max_{X,Y \subset [n]} \frac{\abs{\sum_{x \in X, y \in Y} B(x,y)}}{n^2}.
$$

It is technically easier for us to replace cut norm with a spectral notion of smallness and consider a weighted tournament $A = \frac 1 2 + B$ on $n$ vertices ``close to quasirandom'' if $B$ has small spectral radius, where we use $\frac 1 2 = \frac 1 2 J$ as shorthand when clear from context. That is, letting $\lmax(B)$ be the largest modulus among the eigenvalues of $B$, we require $\lmax(B) \le \eps n$ for some small $\eps > 0$. While this is a quantitatively different notion of quasirandomness than cut norm, it leads to an equivalent definition of a graph being locally tournament Sidorenko or anti-Sidorenko. Since $B$ is skew symmetric the spectral radius coincides with the spectral norm or largest singular value, and this is known to be controlled above and below by the cut norm. That is, since $B$ is skew symmetric with entries in $[-\frac 1 2, \frac 1 2]$, it holds that
$$
n\norm{B}_\Box \le \lmax(B) \le n\sqrt{2\norm{B}_\Box}
$$
by \cite{nikiforov2009cutnormsspectramatrices}*{Theorem 3}. The cut norm is in turn equivalent up to a factor of $2$ to the measure of quasirandom direction used in \cite{fox2024variationssidorenkosconjecturetournaments} to study the forcing property of tournaments. Thus the following definition remains unchanged if one replaces $\lmax(B) \le \eps n$ with $\norm{B}_\Box \le \eps$ or requires that $B$ have $\eps$-quasirandom direction. 

\begin{definition}
    A directed graph $D$ is \textit{locally tournament Sidorenko} (LTS) if there exists $\eps = \eps(D)$ such that for any $n$ and weighted tournament $A = \frac 1 2 + B$ on $n$ vertices with $\lmax(B) \le \eps n$ it holds that
    $$
    t_D(A) \ge 2^{-e(D)},
    $$
    and \textit{locally tournament anti-Sidorenko }(LTAS) if for all such $A$
    $$
    t_D(A) \le 2^{-e(D)}
    $$
    holds.
\end{definition}

\subsection{Characterization of paths}\label{sec:algorithm}

This subsection provides conditions under which a path $D$ is LTS, LTAS, or neither. These conditions involve checking certain signed subgraph counts of $D$. For any directed graphs $D$ and $G$, let $\mathcal{C}_{\textrm{odd}}(G,D)$ be the number of injective functions $f:V(G)\rightarrow V(D)$ such that, for any $(x,y)\in E(G)$, either $(f(x),f(y))\in E(D)$ or $(f(y),f(x))\in E(D)$ and furthermore the number of $(x,y)\in E(G)$ with $(f(y),f(x))\in E(D)$ is odd. Similarly, let $\mathcal{C}_{\textrm{even}}(G,D)$ be the number of injective functions $f:V(G)\rightarrow V(D)$ such that, for any $(x,y)\in E(G)$, either $(f(x),f(y))\in E(D)$ or $(f(y),f(x))\in E(D)$ and the number of $(x,y)\in E(G)$ with $(f(y),f(x))\in E(D)$ is even. Finally, let $\mathcal{C}(G,D)=\mathcal{C}_{\textrm{even}}(G,D)-\mathcal{C}_{\textrm{odd}}(G,D)$. Let $P_n$ be the path oriented left to right on $n$ vertices, $C_n$ be the cycle oriented clockwise on $n$ vertices, and for any directed graph $D$ let $mD$ be the disjoint union of $m$ copies of $D$.

In case $|V(D)| \not\equiv 0 \pmod 4$, we achieve a full characterization that is best summarized in \Cref{alg:paths}.
\begin{algorithm}
\caption{Classify an oriented path $D$ with $|V(D)| \not\equiv 0 \pmod 4$}
\label{alg:paths}
\begin{algorithmic}[1]
\Require An oriented path $D$ with $|V(D)| \not\equiv 0 \pmod 4$.
\Ensure One of \texttt{LTS}, \texttt{LTAS}, or \texttt{Neither}.

\State Compute $\mathcal{C}(P_3, D)$.

\If{$\mathcal{C}(P_3, D) > 0$}
    \State \Return \texttt{LTAS} \Comment{By Theorem~\ref{thm:wedges}}
\ElsIf{$\mathcal{C}(P_3, D) < 0$}
    \State \Return \texttt{LTS} \Comment{By Theorem~\ref{thm:wedges}}
\Else
    \Statex \Comment{$\mathcal{C}(P_3, D) = 0$, so $\ell$ is even and $\ell \equiv 2 \pmod 4$}
    \State Compute $\mathcal{C}(P_5, D)$, $\mathcal{C}(2P_3, D)$ and the minimum $k$ such that $\C(P_{2k+1},D)\not=0$.
    \State Verify that $\mathcal{C}(P_5, D) \neq -\mathcal{C}(2P_3, D)$. \Comment{Guaranteed by Lemma~\ref{lem:P5notequal2P3}}
    \If{($\mathcal{C}(P_5, D) > 0$ and $\mathcal{C}(P_5, D) > -\mathcal{C}(2P_3, D)$) or ($\mathcal{C}(P_5, D) = 0$ and $\mathcal{C}(2P_3, D) > 0$ and ($k$ does not exist or $(-1)^k\C(P_{2k+1},D)>0$))}
        \State \Return \texttt{LTS} \Comment{By Theorem~\ref{thm:countingP5and2P3} and Theorem~\ref{thm:2P3}}
    \ElsIf{($\mathcal{C}(P_5, D) < 0$ and $\mathcal{C}(P_5, D) < -\mathcal{C}(2P_3, D)$) or ($\mathcal{C}(P_5, D) = 0$ and $\mathcal{C}(2P_3, D) < 0$ and ($k$ does not exist or $(-1)^k\C(P_{2k+1},D)<0$))}
        \State \Return \texttt{LTAS} \Comment{By Theorem~\ref{thm:countingP5and2P3} and Theorem~\ref{thm:2P3}}
    \Else
        \State \Return \texttt{Neither} \Comment{By Theorem~\ref{thm:countingP5and2P3} and Theorem~\ref{thm:2P3}}
    \EndIf
\EndIf
\end{algorithmic}
\end{algorithm}
The remainder of this subsection proves the conditions that justify its correctness. The following lemma allows us to write the homomorphism density of a digraph into a tournamenton as a linear combination of the densities of its subgraphs in the corresponding skew-symmetric kernel. This technique has precedent in the undirected setting~\cite{loclov,locfox,locfox}.
\begin{lemma}\label{lem:decomposition}
Let $D$ be a digraph and $W$ be a tournamenton. Decompose $W = \frac{1}{2}J + B$. Then
$$
t_D(W)=2^{-e(D)}+\sum_{D'\subseteq D}2^{-e(D)+e(D')}t_{D'}(B).
$$
where the sum ranges over nonempty subgraphs $D' \subseteq D$ on the same vertex set, ie $V(D) = V(D')$.
\end{lemma}

\begin{proof}
By equation (\ref{eq:hom-density}) and $W=\frac{1}{2}J+B$,
$$
t_D(W)=\int_{[0,1]^{m}}\prod_{(x_1,x_2)\in E(D)}\bigp{\frac{1}{2}+B(x_{1},x_{2})}\prod_{x\in V(D)}dx. 
$$
The desired equation follows by expanding the product above.
\end{proof}

\begin{lemma}\label{lem:path and cycle}
Let $B$ be a skew-symmetric kernel. If $n$ is odd, then $t_{P_{n+1}}(B)=t_{C_n}(B)=0$; if $n\equiv2\mod 4$, then $t_{P_{n+1}}(B)\le 0$ and $t_{C_n}(B)\le 0$; if $n\equiv0\mod 4$, then $t_{P_{n+1}}(B)\ge 0$ and $t_{C_n}(B)\ge 0$.
\end{lemma}

\begin{proof}
Without loss of generality we may assume $B$ is a step kernel corresponding to a skew-symmetric matrix $M_B$, as any kernel can be approximated by a sequence of step kernels. Suppose $M_B$ is an $m$ by $m$ matrix, then $t_{P_{n+1}}(B)=\frac{1}{m^{n+1}}1^\top M_B^n 1$ and $t_{C_n}(B)=\frac{1}{m^{n+1}}\tr(M_B^n)$, where $\tr(\cdot)$ means taking the trace of a matrix. 

When $n$ is odd, $M^n_B$ is skew-symmetric. Thus $t_{P_{n+1}}(B)=\frac{1}{m^{n+1}}1^\top M_B^n 1=0$ and $t_{C_n}(B)=\frac{1}{m^{n+1}}\tr(M_B^n)=0$.

When $n\equiv2\mod 4$, all eigenvalues of $M_B$ are pure imaginary, hence all eigenvalues of $M^n_B$ are non-positive. Thus $M^n_B$ is negative semidefinite, which implies $t_{P_{n+1}}(B)=\frac{1}{m^{n+1}}1^\top M_B^n 1\le 0$ and $t_{C_n}(B)=\frac{1}{m^{n+1}}\tr(M_B^n)\le 0$.

The proof for the case when $n\equiv0\mod 4$ is similar.
\end{proof}

\begin{lemma}\label{lem:P5>2P3}
Let $B$ be a skew-symmetric kernel. Then $t_{P_5}(B)\ge t_{2P_3}(B)$.
\end{lemma}
\begin{proof}
By definition, 
$$
\begin{aligned}
t_{2P_3}(B)&=\int_{[0,1]^6}B(x_1,x_2)B(x_2,x_3)B(x_4,x_5)B(x_5,x_6)\prod_{i=1}^6dx_i\\
&=\left(\int_{[0,1]^3}B(x_1,x_2)B(x_2,x_3)dx_1dx_2dx_3\right)^2\\
&\le \int_{[0,1]}\left(\int_{[0,1]^2}B(x_1,x_2)B(x_2,x_3)dx_1dx_2\right)^2dx_3\\
&=\int_{[0,1]^5}B(x_1,x_2)B(x_2,x_3)B(x_3,x_4)B(x_4,x_5)\prod_{i=1}^5dx_i\\
&=t_{P_5}(B).
\end{aligned}
$$

The inequality above is an application of the {Cauchy-Schwarz inequality} to the integral with respect to $x_3$.
\end{proof}

\begin{lemma}\label{lem:lambda}
Let $D$ be an oriented path. Then there exists $\eps = \eps(D)$ depending only on $D$ such that if $B$ is a skew symmetric $n \times n$ matrix with $\lmax(B) \le \eps n$ then
\begin{itemize}
    \item[(i)] $\abs{\sum_{D' \subseteq D: e(D') \ge 4}2^{-e(D)+e(D')}t_{D'}(B)} < \abs{2^{-e(D)+2}t_{P_3}(B)}.$
    \item[(ii)] $\abs{\sum_{D' \subseteq D: e(D') \ge 6}2^{-e(D)+e(D')}t_{D'}(B)} < \abs{2^{-e(D)+4}t_{P_5}(B)}.$
    \item[(iii)] Let $\cD$ consist of all $D' \subseteq D$ with $e(D') \ge 6$ and such that $D'$ has at least two connected components. Then $\abs{\sum_{D' \in \cD}2^{-e(D)+e(D')}t_{D'}(B)} < \abs{2^{-e(D)+4}t_{2P_3}(B)}.$
    \item[(iv)] For any $\ell \ge 1$, $\abs{\sum_{t=\ell+1}^\infty 2^{-e(D)+2t}t_{P_{2t+1}}(B)} < \abs{2^{-e(D)+2\ell}t_{P_{2\ell+1}}(B)}.$
\end{itemize}
\end{lemma}
\begin{proof}
    By \Cref{lem:path and cycle} the only $D' \subseteq D$ with $t_{D'}(B) \neq 0$ are disjoint unions of paths with an even number of edges. As seen in \Cref{sec:polyexpansion}, when $D'$ is a union of paths, $h_{D'}(B)$ is a polynomial consisting of monomials that look like products of $1^\top B^k 1$ for positive even $k \in \ZZ$. We thus work with this normalization, $h_{D'}(B)$ instead of $t_{D'}(B)$, and start by deriving a spectral formula for such monomials. Let $X_{2t} \coloneqq \abs{1^\top B^{2t}1}$. Since $B$ is skew symmetric we have $\bigp{B^2}^\top = B^\top B^\top = (-B)(-B) = B^2$ so $B^2$ is real symmetric and hence self-adjoint. Then the real spectral theorem guarantees an orthonormal basis of eigenvectors $v_1,\dots,v_n$ with all real entries. Having real entries is important for us, and this is why we do not apply the complex spectral theorem to $B$ itself, which is normal but not self-adjoint. Furthermore the eigenvalues of $B^2$ are clearly $-\lambda_1^2,\dots,-\lambda_n^2$. Then writing $1 = \sum_{i=1}^n c_i v_i$ where $c_i = \biga{1,v_i} \in \RR$ we have
    \begin{align}\label{eq:xformula}
        X_{2t} = \abs{\bigp{\sum_i c_i v_i}^\top B^{2t} \sum_i c_i v_i} = \abs{\sum_i c_i^2 \bigp{-\lambda_i^2}^{t}} = \sum_i c_i^2 \bigp{\lambda_i^2}^t.
    \end{align}
    Note in particular that $c_i \in \RR$ and $\sum_{i=1}^n c_i^2 = 1$.

    Now let $v \coloneqq V(D)$. We first handle inequality $(i)$. Re-normalizing $n^v t_{D'}(B) = h_{D'}(B)$ and using triangle inequality, we may upper bound the left hand side of $(i)$ by a polynomial of the form
    $$
    \sum_{i=1}^{\infty} a_i \prod_{l=0}^{\infty} X_{2l}^{m_{i,l}},
    $$
    where $m_{i,l} \in \ZZ_{\ge 0}$ is such that $m_{i,0} < v-3$ for all $i$ ($m_{i,0} \in \set{0,v-3}$ correspond to $D'$ with $e(D') < 4$) and $a_i = a_i(D) \in \QQ$ are constants depending only on $D$. Explicitly $a_i$ is the normalizing power of $\frac 1 2$ in front of the monomial (depending on its degree) multiplied by the coefficient given by signed subpath counts, but since we only aim to find some $\eps(D)$ we don't need to worry about their actual values. It is important to note that at most a bounded (in terms of $D$) number of the $a_i$ can be nonzero, in particular only those $i$ such that $\sum_{l=0}^\infty m_{i,l}(2l+1) = v$. Thus going forward assume any monomials we are handling satisfy $\sum_{l=0}^\infty m_{i,l}(2l+1) = v$ and let $c = c(D)$ denote the number of these monomials. 
    
    To establish $(i)$ it suffices to show that
    \begin{equation}\label{eq:reduction}
        \prod_{l=0}^{\infty} X_{2l}^{m_{i,l}} \le c' X_0^{v-3}X_2 = c' n^{v-3}X_2
    \end{equation}
    for any $i$, where $c' = c'(D)$. Explicitly $c' = (2^{v-2}a_{i_*}c)^{-1}$ where $a_{i_*}$ has the largest magnitude of the $a_i$, but again we will not care about this. Indeed this inequality would establish that all higher order terms (at most $c$ of them) are uniformly small enough that even when scaled by their coefficients and added up (regardless of signs) they are still at most $\frac{n^{v-3}}{2^{v-2}}X_2 = 2^{-v+2}\abs{t_{P_3}(B)}$ as desired.
    
    Letting $m_{i,l} = m_l$ for any $i$ such that we want to show \eqref{eq:reduction}, by \Cref{eq:xformula} we have
    \begin{align*}
         \prod_{l=0}^{\infty} X_{2l}^{m_{l}} &= n^{m_{0}}\prod_{l=1}^{\infty} X_{2l}^{m_{l}} = n^{m_{0}}\prod_{l=1}^{\infty} \bigp{\sum_{j=1}^n c_j^2 \lambda_j^{2l}}^{m_{l}} \\
         &\le n^{m_{0}}\prod_{l=1}^{\infty} \bigp{\lmax^{2l-2}\sum_{j=1}^n c_j^2 \lambda_j^{2}}^{m_{l}} = n^{m_{0}}\prod_{l=1}^{\infty} \lmax^{m_l(2l-2)}X_2^{m_l}\\
         &= n^{m_{0}} \lmax^{\sum_{l=1}^\infty m_l(2l-2)} X_2^{\sum_{l=1}^\infty m_l}.
    \end{align*}
    Then, letting $S \coloneqq \sum_{l=1}^\infty m_l$, if $\lmax \le \eps n$ this is at most
    $$
    \eps^{v-m_0-3S}n^{v-3S}X_2^{S} \le \eps^{v-m_0-3S}n^{v-3}X_2,
    $$
    using the fact that $X_2 \le n^3$ by definition (also note $S \ge 1$ by assumption that $m_0 < v$). Note that the exponent $v-m_0-3S = \sum_{l=1}^\infty 2m_l(l-1) \ge 1$ since it is only equal to $0$ in the two cases for which we do not need to prove \eqref{eq:reduction}. Thus setting
    $$
    \eps = \eps(D) = \frac{c'(D)}{2}
    $$
    suffices to prove inequality $(i)$.

    Inequalities $(ii)-(iv)$ follow identically, with small twists. For instance in $(ii)$ we aim to show
    \begin{equation}\label{eq:reduction2}
        \prod_{l=0}^{\infty} X_{2l}^{m_{i,l}} \le c' X_0^{v-5}X_4 = c' n^{v-5}X_4
    \end{equation}
    for monomials with $m_{i,0} \le v-5$. For any such monomial $\prod_{l=0}^{\infty} X_{2l}^{m_{l}}$, if $m_l > 0$ for any $l \ge 2$ then simply do as above, pulling out $\lmax^{m_l(2l-4)}$ and bounding all other factors trivially in terms of $n$. If not, then because $m_{0} \le v-5$, we must have $m_{1} \ge 2$. Then bounding $X_2^2 \le X_4$ by Cauchy-Schwarz (re-normalized \Cref{lem:P5>2P3}) and the other factors trivially in terms of $n$ gives \Cref{eq:reduction2}. For $(iii)$ since each connected component of $D' \in \cD$ has at least two edges one may pull out powers of $\lmax$ until reaching $X_2^2$. $(iv)$ also follows by pulling out powers of $\lmax$.
\end{proof}

A kernel $B$ is said to be \textit{balanced} if, for almost all $x\in [0,1]$, $\int_{[0,1]}B(x,y)dy=0$.

\begin{lemma}\label{lem:balanced}
Let $B$ be a balanced skew-symmetric kernel and let $T$ be any digraph whose underlying graph has a vertex with degree 1. Then we have $t_T(B)=0$.
\end{lemma}
\begin{proof}
Suppose the vertex set of $T$ is $\{v_1,\dots,v_m\}$, the vertex with degree 1 is $v_1$ and it has an in-neighbor $v_2$. By definition,
$$
\begin{aligned}
t_T(B)&=\int_{[0,1]^m}\prod_{(v_j,v_k)\in E(T)}B(x_j,x_k)\prod_{i=1}^mdx_i\\
&=\int_{[0,1]^{m-1}}\prod_{\substack{(v_j,v_k)\in E(T)\\ (j,k)\not=(2,1)}}B(x_j,x_k)\left(\int_{[0,1]}B(x_2,x_1)dx_1\right)\prod_{i=2}^mdx_i = 0.
\end{aligned}
$$
Here the third equality makes use of the fact that $B$ is balanced.
\end{proof}

\begin{lemma}\label{lem:balanced_P3}
Let $B$ be a skew-symmetric kernel. Then $B$ is balanced if and only if $t_{P_3}(B)=0$.
\end{lemma}
\begin{proof}
If $B$ is balanced, by Lemma \ref{lem:balanced} we have $t_{P_3}(B)=0$. On the other hand, if $t_{P_3}(B)=0$, by definition
$$
t_{P_3}(B)=\int_{[0,1]^3}B(x,y)B(y,z)dxdydz=-\int_{[0,1]}\left(\int_{[0,1]}B(y,x)dx\right)^2dy\le 0.
$$
Note that the inequality is tight only if $\int_{[0,1]}B(y,x)dx=0$ for almost all $y$. Thus $B$ is balanced.
\end{proof}

The following theorem provides a simple sufficient condition for an oriented path to be LTS/LTAS.

\begin{theorem}\label{thm:wedges}
Let $D$ be an oriented path. If $\mathcal{C}(P_3,D)>0$, then $D$ is LTAS. On the other hand, if $\mathcal{C}(P_3,D)<0$, then $D$ is LTS.
\end{theorem}

\begin{proof}
Let $\eps = \eps(D) > 0$ be as in \Cref{lem:lambda} and consider any $n$-vertex weighted tournament $A=\frac{1}{2}+B$ such that  $\lmax(B)\le \varepsilon n$. By Lemma \ref{lem:decomposition} together with the fact that $t_{K_2}(B)=t_{2K_2}(B)=0$,
$$
t_D(A)=2^{-e(D)}+ \mathcal{C}(P_3,D)2^{-e(D)+2}t_{P_3}(B)+\sum_{\substack{D'}}2^{-e(D)+e(D')}t_{D'}(B),
$$
where the sum ranges over all subgraph $D'\subseteq D$ such that $D'$ is a disjoint union of paths with even number of edges with $e(D')\ge 4$. By \Cref{lem:path and cycle}, we have $t_{P_3}(B)\le 0$. Then inequality $(i)$ of \Cref{lem:lambda} guarantees that the tail is small in magnitude, and hence $t_D(A)\le 2^{-e(D)}$ if $\mathcal{C}(P_3,D)>0$ and $t_D(A)\ge 2^{-e(D)}$ if $\mathcal{C}(P_3,D)<0$.
\end{proof}

It remains to consider when $\mathcal{C}(P_3,D)=0$. We will use what has become known as the tensor power trick.

\begin{definition}[Tensor product of kernels]
Let $W_1,W_2$ be kernels on $[0,1]^2$. Their \emph{tensor product} is the kernel
\[
  W_1\otimes W_2 : \bigl([0,1]^2\bigr)^2 \longrightarrow [0,1],
  \qquad
  ( (x_1,x_2),(y_1,y_2) ) \longmapsto W_1(x_1,y_1)\, W_2(x_2,y_2).
\]
We will identify $[0,1]^2$ with $[0,1]$ via any
measure-preserving bijection if we wish to view $W_1\otimes W_2$ again as a kernel on $[0,1]^2$.
\end{definition}

The following proposition is straightforward from the definition; the proof is omitted.
\begin{proposition}\label{prop:tensor}
Let $W_1,W_2$ be kernels. For every finite simple graph $F$,
  \[
     t\bigl(F,\, W_1\otimes W_2\bigr) \;=\; t_{F}(W_1)\, t_{F}(W_2).
  \]
\end{proposition}

For a kernel $W:[0,1]^2\to\mathbb{R}$ and an integer $n\ge1$, we define the
$n$-fold tensor product
\[
   W^{\otimes n}
   \;\coloneqq\; \underbrace{W\otimes W\otimes\cdots\otimes W}_{n\text{ times}}.
\]

\begin{theorem}\label{thm:countingP5and2P3}
Let $D$ be an oriented path with $\mathcal{C}(P_3,D)=0$ and \\ $\mathcal{C}(P_5,D)\not\in\{0, -\mathcal{C}(2P_3,D)\}$.
\begin{itemize}
    \item[(i)]If $\mathcal{C}(P_5,D)>0$ and $\mathcal{C}(P_5,D)> -\mathcal{C}(2P_3,D)$, then $D$ is LTS.
    \item[(ii)]If $\mathcal{C}(P_5,D)<0$ and $\mathcal{C}(P_5,D)< -\mathcal{C}(2P_3,D)$, then $D$ is LTAS.
    \item [(iii)] Otherwise, i.e. if 
    \begin{itemize}
        \item[$\bullet$] $\mathcal{C}(P_5,D)>0$ and $\mathcal{C}(P_5,D)< -\mathcal{C}(2P_3,D)$, or
        \item[$\bullet$] $\mathcal{C}(P_5,D)<0$ and $\mathcal{C}(P_5,D)> -\mathcal{C}(2P_3,D)$,
    \end{itemize}
    then $D$ is neither LTS nor LTAS.
\end{itemize}
\end{theorem}

\begin{proof}
Let $\eps = \eps(D) > 0$ be as in \Cref{lem:lambda} and consider any $n$-vertex weighted tournament $A=\frac{1}{2}+B$ such that  $\lmax(B)\le \varepsilon n$. By \Cref{lem:decomposition} and since $\mathcal{C}(P_3,D)=0$, we have
\begin{equation}\label{eq:decomp_path}
t_D(A)=2^{-e(D)}+(\mathcal{C}(P_5,D)t_{P_5}( B)+\mathcal{C}(2P_3,D)t_{2P_3}( B))2^{-e(D)+4}+\sum_{\substack{D'}}2^{-e(D)+e(D')}t_{D'}(B),
\end{equation}
where the sum ranges over all subgraphs $D'\subseteq D$ such that $D'$ is a disjoint union of paths with even number of edges with $e(D')\ge 6$. 

For \textit{(i)}, we assume $\mathcal{C}(P_5,D)>0$ and $\mathcal{C}(P_5,D) + \mathcal{C}(2P_3,D) > 0$. Then by \Cref{eq:decomp_path} and \Cref{lem:P5>2P3}
\begin{align*}
    t_D(A) &\ge 2^{-e(D)}+(\mathcal{C}(P_5,D)+\min\{\mathcal{C}(2P_3,D),0\})2^{-e(D)+4}t_{P_5}(B) -  \abs{\sum_{\substack{D'}}2^{-e(D)+e(D')}t_{D'}(B)}\\
    &\ge 2^{-e(D)},
\end{align*}
since the tail is small by inequality $(ii)$ of \Cref{lem:lambda}. Thus, we conclude that $D$ is LTS. The proof for \textit{(ii)} is similar.

For \textit{(iii)}, we first assume $\mathcal{C}(P_5,D)>0$ and $\mathcal{C}(P_5,D) + \mathcal{C}(2P_3,D) < 0$. Let $B_1$ be the step kernel corresponding to the matrix
$$
\begin{bmatrix}
0 & 1\\
-1& 0
\end{bmatrix}.
$$
It is easy to check that $t_{P_5}(B_1)=t_{2P_3}(B_1)=\frac{1}{16}$. For $\varepsilon>0$, define $W_{1,\varepsilon} \coloneqq \frac{1}{2}J+\varepsilon B_1$. Then by \Cref{eq:decomp_path}, for sufficiently small $\eps$
$$
t_D(W_{1,\varepsilon})=2^{-e(D)}+(\mathcal{C}(P_5,D)+\mathcal{C}(2P_3,D))\frac{1}{16}2^{-e(D)+4}\varepsilon^4+O(\varepsilon^6)<2^{e(D)}.
$$
$D$ is not LTS. 

For the other direction, let $B'$ be the step kernel corresponding to the matrix
$$
\begin{bmatrix}
0 & 1& -1\\
-1& 0& 0\\
1& 0 & 0
\end{bmatrix}.
$$ 
It is easy to check that $t_{P_5}(B')=\frac{1}{8}$ and $t_{2P_3}(B')=\frac{1}{16}$. Let $m$ be a sufficiently large integer and let $B_2=B'^{\otimes m}$. Then by \Cref{prop:tensor}, we have $t_{P_5}(B_2)=\frac{1}{8^m}$ and $t_{2P_3}(B_2)=\frac{1}{16^m}$, so for a large enough choice of $m$
$$
\mathcal{C}(P_5,D)\frac{1}{8^m}+\mathcal{C}(2P_3,D)\frac{1}{16^m} > 0.
$$
Define $W_{2,\varepsilon} \coloneqq \frac{1}{2}J+\varepsilon B_2$. Then by \Cref{eq:decomp_path}, for sufficiently small $\eps$
$$
t_D(W_{2,\varepsilon})=2^{-e(D)}+\bigp{\mathcal{C}(P_5,D)\frac{1}{8^m}+\mathcal{C}(2P_3,D)\frac{1}{16^m}}2^{-e(D)+4}\varepsilon^4+ O(\varepsilon^6)>2^{-e(D)},
$$
so $D$ is not LTAS. The proof for the case when $\mathcal{C}(P_5,D)<0$ and $\mathcal{C}(P_5,D) + \mathcal{C}(2P_3,D) > 0$ is similar.
\end{proof}

By \Cref{thm:countingP5and2P3}, there exist oriented paths that are neither LTS nor LTAS. For example, let $D_1$ be the oriented path on $10$ vertices 
$\rightarrow\rightarrow\rightarrow\rightarrow\rightarrow\leftarrow\rightarrow\leftarrow\rightarrow$. 
Then $\mathcal{C}(P_3, D_1) = 0$, $\mathcal{C}(P_5, D_1) = 2$, and $\mathcal{C}(2P_3, D_1) = -11$. 
By \Cref{thm:countingP5and2P3}, $D_1$ is neither LTS nor LTAS. 
More generally, we obtain the following result.

\begin{corollary}\label{cor:neither}
For any $k \ge 5$, the oriented path whose first $k$ edges are directed $\rightarrow$, followed by $k-1$ edges alternating between $\leftarrow$ and $\rightarrow$, is neither LTS nor LTAS.
\end{corollary}

We consider now when $\mathcal{C}(P_3,D)=\mathcal{C}(P_5,D)=0$ and $\mathcal{C}(2P_3,D)\not=0$. The following lemma will be helpful for constructions.

\begin{lemma}\label{lem:construction_for_simple_t(P)}
For any $a,b \in [0,1]$ with $2\sqrt{b}(\sqrt{a}+\sqrt{1-a})<1/2$, there exists a tournamenton $W=\frac{1}{2}+B$ such that $\norm{B}_\infty \le 4\sqrt{b}$ and, for any integer $k\ge1$, $t(P_{2k+1},B)=(-1)^{k}ab^k$.
\end{lemma}

\begin{proof}
Let $v_1=\frac{1}{2}(1,1,1,1)$, $v_2=\frac{1}{2}(1,1,-1,-1)$ and $w=\frac{1}{2}(1,-1,1,-1)$ be row vectors. It is clear that $v_1,v_2$ and $w$ are pairwise orthonormal. Let $u=\sqrt{a}v_1+\sqrt{1-a}v_2$. Let $J$ be the all-1 $4\times 4$ matrix. Define the skew-symmetric matrix
$$
B:=4\sqrt{b}(u^\top w-w^\top u),
$$
and let $W=\frac{1}{2}J+B$. Note that the entries of $B$ have absolute value at most $2\sqrt{b}(\sqrt{a}+\sqrt{1-a})<1/2$, thus $W$ has nonnegative entries and we may view $W$ as a tournamenton. Finally,
$$
t(P_{2k+1},B)=\frac{1^\top B^{2k}1}{4^{2k+1}}=\frac{(-1)^{k}b^{k}}{4}1^\top (u^\top u+w^\top w)1=(-1)^{k}ab^{k}.
$$
\end{proof}

\begin{theorem}\label{thm:2P3}
Let $D$ be an oriented path with $\mathcal{C}(P_3,D)=\mathcal{C}(P_5,D)=0$ and $\mathcal{C}(2P_3,D)\not=0$. Let $\ell$ be the minimum integer such that $\mathcal{C}(P_{2\ell+1},D)\not=0$, if it exists.
\begin{itemize}
    \item[(i)]If $\mathcal{C}(2P_3,D)>0$ and 
    \begin{itemize}
        \item[$\bullet$] $\ell$ does not exist, or
        \item[$\bullet$] $(-1)^\ell\mathcal{C}(P_{2\ell+1},D)>0$,
    \end{itemize}
    then $D$ is LTS.
    \item[(ii)]If $\mathcal{C}(2P_3,D)<0$ and 
    \begin{itemize}
        \item[$\bullet$] $\ell$ does not exist, or
        \item[$\bullet$] $(-1)^\ell\mathcal{C}(P_{2\ell+1},D)<0$,
    \end{itemize}
    then $D$ is LTAS.
    \item [(iii)] Otherwise, i.e. if $\ell$ exists and $\mathcal{C}(2P_3,D)(-1)^\ell\mathcal{C}(P_{2\ell+1},D)<0$,
    then $D$ is neither LTS nor LTAS.
\end{itemize}
\end{theorem}

\begin{proof}
Let $\eps = \eps(D) > 0$ be as in \Cref{lem:lambda} and consider any $n$-vertex weighted tournament $A=\frac{1}{2}+B$ such that  $\lmax(B)\le \varepsilon n$. By \Cref{lem:decomposition} and since $\mathcal{C}(P_3,D)=\mathcal{C}(P_5,D)=0$, we have
\begin{equation}\label{eq:decomp_path2}
t_D(A)=2^{-e(D)}+\mathcal{C}(2P_3,D)t_{2P_3}( B)2^{-e(D)+4}+\sum_{\substack{D'}}2^{-e(D)+e(D')}t_{D'}(B),
\end{equation}
where the sum ranges over all subgraphs $D'\subseteq D$ such that $D'$ is a disjoint union of paths with even number of edges with $e(D')\ge 6$.

For \textit{(i)}, we assume $\mathcal{C}(2P_3,D)>0$. 

If $\ell$ does not exist, then
\begin{align*}
    t_D(A) &\ge 2^{-e(D)}+\mathcal{C}(2P_3,D)2^{-e(D)+4}t_{2P_3}(B) -  \abs{\sum_{\substack{D' \in \cD}}2^{-e(D)+e(D')}t_{D'}(B)}\\
    &\ge 2^{-e(D)}.
\end{align*}
where $\cD$ is as in inequality $(iii)$ of \Cref{lem:lambda} .

If $\ell$ exists and $(-1)^\ell\mathcal{C}(P_{2\ell+1},D)>0$, then
\begin{align*}
t_D(A)&\ge 2^{-e(D)}+\mathcal{C}(2P_3,D)2^{-e(D)+4}t_{2P_3}(B) -  \abs{\sum_{\substack{D' \in \cD}}2^{-e(D)+e(D')}t_{D'}(B)} \\
&+ (-1)^\ell\mathcal{C}(P_{2\ell+1},D)2^{-e(D)+2\ell}t_{P_{2\ell+1}}(B) -  \abs{\sum_{t=\ell + 1}^\infty 2^{-e(D)+2t}t_{P_{2t+1}}(B)}\\
&\ge 2^{-e(D)}.
\end{align*}
by inequalities $(iii)$ and $(iv)$ of \Cref{lem:lambda} . Thus in either case we conclude that $D$ is LTS. The proof for \textit{(ii)} is similar.

For \textit{(iii)}, without loss of generality, assume $\C(2P_3,D)>0$, then $(-1)^\ell\mathcal{C}(P_{2\ell+1},D)<0$. Applying \Cref{lem:construction_for_simple_t(P)} with $a=1$, for any $b > 0$ there exists a tournamenton $W_1=\frac{1}{2}+B_1$ such that $\norm{B_1}_\infty \le 4\sqrt{b}$ and $t(P_{2k+1},B_1)=(-1)^kb^{k}$. By \Cref{eq:decomp_path2},
$$
t_D(W_1)=2^{-e(D)}+\C(2P_3,D)2^{-e(D)+4}b^{2}- O(b^{3})>2^{-e(D)}.
$$
for sufficiently small $b$, and thus $D$ is not LTAS.

Applying \Cref{lem:construction_for_simple_t(P)} with $a=b^{\ell-1}$, there exists a tournamenton $W_2=\frac{1}{2}+B_2$ such that $t(P_{2k+1},B_2)=(-1)^kb^{k+\ell-1}$. By \Cref{eq:decomp_path2},
$$
t_D(W_2)=2^{-e(D)}+(-1)^{\ell}\C(P_{2\ell+1},D)2^{-e(D)+2\ell}b^{2\ell-1}+O(b^{2\ell})<2^{-e(D)}.
$$
for sufficiently small $b$, and thus $D$ is not LTS. This completes the proof.
\end{proof}

\begin{lemma}\label{lem:P5notequal2P3}
Let $D$ be an oriented path on $\ell$ vertices. If $\ell$ is odd, then $\mathcal{C}(P_3,D)\not=0$. If $\ell\equiv 2 \mod 4$, then $\mathcal{C}(P_5,D)\not=-\mathcal{C}(2P_3,D)$. 
\end{lemma}

\begin{proof}
It suffices to show that, if $\ell$ is odd, then $\mathcal{C}(P_3,D)$ is odd, and if $\ell\equiv 2 \mod 4$, then $\mathcal{C}(P_5,D)$ is even and $\mathcal{C}(2P_3,D)$ is odd. Without loss of generality we can assume that $D$ is the directed path $P_\ell$. Then $\mathcal{C}(P_3,D)=\ell-2$ is odd if $\ell$ is odd. Moreover, $\mathcal{C}(P_5,D)=\ell-4$ is even and $\mathcal{C}(2P_3,D)=\binom{\ell-4}{2}$ is odd when $\ell\equiv 2 \mod 4$.
\end{proof}

By Theorems \ref{thm:wedges}, \ref{thm:countingP5and2P3}, \ref{thm:2P3} and \Cref{lem:P5notequal2P3}, we have a full characterization of the local Sidorenko properties for oriented paths on $\ell$ vertices when $\ell$ is not a multiple of 4 (\Cref{alg:paths}). A case not covered by the above is when $\mathcal{C}(P_3, D) = 0$ and $\mathcal{C}(P_5, D)=\mathcal{C}(2P_3,D)$. For example, let $D_2$ be the oriented path
$\rightarrow\rightarrow\leftarrow\rightarrow\rightarrow\leftarrow\rightarrow$. 
Then $\mathcal{C}(P_3, D_2) = 0$ and $\mathcal{C}(P_5, D_2) = -\mathcal{C}(2P_3, D_2) = -2$. Thus $D_2$ cannot be classified using \Cref{thm:wedges} or \Cref{thm:countingP5and2P3}. One may continue to consider terms farther in the tail, but we do not pursue this.

\subsection{Characterization of cycles}\label{sec:cycle}
Let \( C_\ell \) be the cycle of length $\ell$ oriented clockwise, and \( C \) be obtained by reversing the orientation of \( t \) edges of \( C_\ell \). Then, by \Cref{lem:decomposition} and \Cref{lem:path and cycle}, for any tournamenton \( W = \tfrac{1}{2}J + B \), we have  
\begin{equation}\label{eq:cycle}
t_{C}(W) = 2^{-\ell} + \sum_Q \mathcal{C}(Q,C)2^{-\ell+e(Q)}t_{Q}( B) + (-1)^tt_{C_\ell}( B),
\end{equation}
where the sum ranges over all disjoint unions $Q$ of paths with even number of edges. 
The analysis for oriented cycles proceeds similarly to that for oriented paths, except for the additional term \((-1)^tt_{C_\ell}( B)\) in the expansion of \(t_{C}( W)\). 
\begin{theorem}\label{thm:cycles NOT TS/TAS}
Let $\ell\ge 4$ be an even integer and let $0\le t\le \ell$.
Let \( C \) be an oriented cycle obtained by reversing the orientation of \( t \) edges of \( C_\ell \).
\begin{itemize}
    \item[(i)] If 
    \begin{itemize}
        \item[$\bullet$] $\ell\equiv 0 \pmod{4}$ and $t$ is even, or
        \item[$\bullet$] $\ell\equiv 2 \pmod{4}$ and $t$ is odd,
    \end{itemize}
    then $C$ is not LTAS.
    \item[(ii)] If 
    \begin{itemize}
        \item[$\bullet$] $\ell\equiv 0 \pmod{4}$ and $t$ is odd, or
        \item[$\bullet$] $\ell\equiv 2 \pmod{4}$ and $t$ is even,
    \end{itemize}
    then $C$ is not LTS.
\end{itemize}
\end{theorem}
\begin{proof}
Let $B$ be the balanced step kernel corresponding to the matrix
$$
M_B=
\begin{bmatrix}
0 & 1& -1\\
-1& 0& 1\\
1& -1& 0
\end{bmatrix}.
$$
The eigenvalues of $M_B$ are 0 and $\pm \sqrt{3} i$, each of multiplicity one. Thus, for integer $k\ge1$, $$t_{C_{4k}}( B)=2\cdot3^{2k}\cdot 3^{-4k}=\frac{2}{3^{2k}}>0$$
and
$$t_{C_{4k+2}}( B)=-2\cdot3^{2k+1}\cdot 3^{-4k-2}=-\frac{2}{3^{2k+1}}<0.$$

For $\varepsilon>0$ define $W_{\varepsilon} \coloneqq \frac{1}{2}J+\varepsilon B$. By \Cref{lem:decomposition} and \Cref{lem:balanced} we have
$$
t_{C}( W_{\varepsilon})=2^{-\ell}+ \eps^\ell t_{C}( B)=2^{-\ell}+(-1)^t \eps^\ell t_{C_{\ell}}( B).
$$
Thus, if $\ell=4k$ and $t$ is even, for any $\eps > 0$ we have $t_{C}(W_\varepsilon) > 2^{-\ell}$ and hence $D$ is not LTAS. The proofs for the other cases are similar.
\end{proof}

Using \Cref{thm:cycles NOT TS/TAS} and \Cref{eq:cycle}, one may modify the arguments in \Cref{sec:algorithm} to obtain sufficient conditions, complementing the above necessary ones. As the ideas are essentially the same, we state the results and omit proof details. The following is a variant of \Cref{thm:wedges}.

\begin{theorem}\label{thm:wedges_cycle}
Let integers $\ell\ge t\ge 0$ and let \( C \) be an oriented cycle obtained by reversing the orientation of \( t \) edges of \( C_\ell \) such that $\mathcal{C}(P_3,C)\not=0$.
\begin{itemize}
    \item[(i)] If $\mathcal{C}(P_3,C)<0$, and
    \begin{itemize}
        \item[$\bullet$] $\ell$ is odd, or
        \item[$\bullet$] $\ell\equiv 0 \pmod{4}$ and $t$ is even, or
        \item[$\bullet$] $\ell\equiv 2 \pmod{4}$ and $t$ is odd,
    \end{itemize}
    then $C$ is LTS.
    \item[(ii)] If $\mathcal{C}(P_3,C)>0$, and
    \begin{itemize}
        \item[$\bullet$] $\ell$ is odd, or
        \item[$\bullet$] $\ell\equiv 0 \pmod{4}$ and $t$ is odd, or
        \item[$\bullet$] $\ell\equiv 2 \pmod{4}$ and $t$ is even,
    \end{itemize}
    then $C$ is LTAS.
    \item[(iii)] Otherwise, $C$ is neither LTS nor LTAS.
\end{itemize}
\end{theorem}
\begin{proof}
    Under the assumptions of $(i),(ii)$, by \Cref{lem:path and cycle} the cycle term contributes in the same direction as the $P_3$ term (or is zero); the proof of \Cref{thm:wedges} directly gives the result. For case $(iii)$, \Cref{thm:cycles NOT TS/TAS} gives one direction. For the other direction take any kernel $B$ with $t(P_3,B) > 0$ and then $\eps B$ for $\eps \to 0$.
\end{proof}

\Cref{thm:wedges_cycle} handles the case when $\mathcal{C}(P_3,C)\not=0$. When $\mathcal{C}(P_3,C)=0$ and $\mathcal{C}(P_5,C)\not\in\{0,-\mathcal{C}(2P3,C)\}$, we have the following variant of \Cref{thm:countingP5and2P3}.

\begin{theorem}\label{thm:countingP5and2P3_cycle}
Let integers $\ell\ge t\ge 0$ and let \( C \) be an oriented cycle obtained by reversing the orientation of \( t \) edges of \( C_\ell \) such that $\mathcal{C}(P_3,C)=0$ and $\mathcal{C}(P_5,C)\not\in\{0, -\mathcal{C}(2P_3,C)\}$.
\begin{itemize}
    \item[(i)]If $\mathcal{C}(P_5,C)>0$ and $\mathcal{C}(P_5,C)> -\mathcal{C}(2P_3,C)$, and
    \begin{itemize}
        \item[$\bullet$] $\ell$ is odd, or
        \item[$\bullet$] $\ell\equiv 0 \pmod{4}$ and $t$ is even, or
        \item[$\bullet$] $\ell\equiv 2 \pmod{4}$ and $t$ is odd,
    \end{itemize}
    then $D$ is LTS.
    \item[(ii)]If $\mathcal{C}(P_5,C)<0$ and $\mathcal{C}(P_5,C)< -\mathcal{C}(2P_3,C)$, and
    \begin{itemize}
        \item[$\bullet$] $\ell$ is odd, or
        \item[$\bullet$] $\ell\equiv 0 \pmod{4}$ and $t$ is odd, or
        \item[$\bullet$] $\ell\equiv 2 \pmod{4}$ and $t$ is even,
    \end{itemize}
    then $D$ is LTAS.
    \item[(iii)] Otherwise, $C$ is neither LTS nor LTAS.
\end{itemize}
\end{theorem}
\begin{proof}
    Identical to \Cref{thm:wedges_cycle}, replacing $P_3$ with $P_5$.
\end{proof}

The following theorem handles the case when $\mathcal{C}(P_3,C)=\mathcal{C}(P_5,C)=0$ and $\mathcal{C}(2P_3,C)\not=0$, which is a variant of \Cref{thm:2P3}.

\begin{theorem}\label{thm:2P3_cycle}
Let integers $\ell\ge t\ge 0$ and let \( C \) be an oriented cycle obtained by reversing the orientation of \( t \) edges of \( C_\ell \) with $\mathcal{C}(P_3,C)=\mathcal{C}(P_5,C)=0$ and $\mathcal{C}(2P_3,C)\not=0$. Let $k$ be the minimum integer such that $\mathcal{C}(P_{2k+1},C)\not=0$.
\begin{itemize}
    \item[(i)]If $\mathcal{C}(2P_3,C)>0$ and 
    \begin{itemize}
        \item[$\bullet$] $k$ does not exist, or
        \item[$\bullet$] $(-1)^k\mathcal{C}(P_{2k+1},C)>0$,
    \end{itemize}
    and
    \begin{itemize}
        \item[$\bullet$] $\ell$ is odd, or
        \item[$\bullet$] $\ell\equiv 0 \pmod{4}$ and $t$ is even, or
        \item[$\bullet$] $\ell\equiv 2 \pmod{4}$ and $t$ is odd,
    \end{itemize}
    then $D$ is LTS.
    \item[(ii)]If $\mathcal{C}(2P_3,C)<0$ and 
    \begin{itemize}
        \item[$\bullet$] $k$ does not exist, or
        \item[$\bullet$] $(-1)^k\mathcal{C}(P_{2k+1},C)<0$,
    \end{itemize}
    and
    \begin{itemize}
        \item[$\bullet$] $\ell$ is odd, or
        \item[$\bullet$] $\ell\equiv 0 \pmod{4}$ and $t$ is odd, or
        \item[$\bullet$] $\ell\equiv 2 \pmod{4}$ and $t$ is even,
    \end{itemize}
    then $D$ is LTAS.
    \item [(iii)] Otherwise, $D$ is neither LTS nor LTAS.
\end{itemize}
\end{theorem}
\begin{proof}
     Identical to \Cref{thm:wedges_cycle}, replacing $P_3$ with $2P_3$.
\end{proof}

We also have a variant of \Cref{lem:P5notequal2P3} for oriented cycles.
\begin{lemma}\label{lem:P5notequal2P3_cycle}
Let $C$ be an oriented cycle on $\ell$ vertices. If $\ell$ is odd, then $\mathcal{C}(P_3,C)\not=0$. If $\ell\equiv 2 \mod 4$, then $\mathcal{C}(P_5,C)\not=-\mathcal{C}(2P_3,C)$. 
\end{lemma}

Now by Theorems \ref{thm:wedges_cycle}, \ref{thm:countingP5and2P3_cycle}, \ref{thm:2P3_cycle} and \Cref{lem:P5notequal2P3_cycle}, we have \Cref{alg:LTS_LTAS_cycle} which characterizes the locally Sidorenko property of an oriented cycle $C$ with $|V(C)| \not\equiv 0 \pmod 4$.

\begin{algorithm}
\caption{Classify an oriented cycle \(C\) with $|V(C)| \not\equiv 0 \pmod 4$}
\label{alg:LTS_LTAS_cycle}
\begin{algorithmic}[1]
\Require An oriented cycle \(C\) obtained from \(C_\ell\) by reversing \(t\) edges, where \(\ell \not\equiv 0 \pmod 4\).
\Ensure One of \(\texttt{LTS}\), \(\texttt{LTAS}\), or \(\texttt{Neither}\).

\State Compute \(\mathcal{C}(P_3,C)\).

\If{\(\mathcal{C}(P_3,C)\neq 0\)} 
    \Comment{Apply Theorem \ref{thm:wedges_cycle}}
    \If{ \(\mathcal{C}(P_3,C)<0\) 
         \textbf{and} (\(\ell\) odd \textbf{or} (\(\ell\equiv2\pmod4\) and \(t\) odd)) }
        \State \Return \(\texttt{LTS}\)
    \ElsIf{ \(\mathcal{C}(P_3,C)>0\) 
             \textbf{and} (\(\ell\) odd \textbf{or} (\(\ell\equiv2\pmod4\) and \(t\) even)) }
        \State \Return \(\texttt{LTAS}\)
    \Else
        \State \Return \(\texttt{Neither}\)
    \EndIf
\Else
    \Comment{\(\mathcal{C}(P_3,C)=0\); apply Theorem \ref{thm:countingP5and2P3_cycle}}
    \State Compute \(\mathcal{C}(P_5,C)\), \(\mathcal{C}(2P_3,C)\), and minimum $k$ such that $\C(P_{2k+1},C)\not=0$.
    \State Verify that $\mathcal{C}(P_5, D) \neq -\mathcal{C}(2P_3, D)$. \Comment{Guaranteed by Lemma~\ref{lem:P5notequal2P3_cycle}}
    \If{
        \(
        (\mathcal{C}(P_5,C)>0 \land \mathcal{C}(P_5,C)> -\mathcal{C}(2P_3,C))
        \lor
        (\mathcal{C}(P_5,C)=0 \land \mathcal{C}(2P_3,C)>0\land(k~\text{does not exist}\lor (-1)^k\C(P_{2k+1},C)>0))
        \)
        \textbf{and}
        (\(\ell\) odd \textbf{or} (\(\ell\equiv2\pmod4\) and \(t\) odd))
       }
        \State \Return \(\texttt{LTS}\)
    \ElsIf{
        \(
        (\mathcal{C}(P_5,C)<0 \land \mathcal{C}(P_5,C)< -\mathcal{C}(2P_3,C))
        \lor
        (\mathcal{C}(P_5,C)=0 \land \mathcal{C}(2P_3,C)<0\land(k~\text{does not exist}\lor (-1)^k\C(P_{2k+1},C)<0))
        \)
        \textbf{and}
        (\(\ell\) odd \textbf{or} (\(\ell\equiv2\pmod4\) and \(t\) even))
       }
        \State \Return \(\texttt{LTAS}\)
    \Else
        \State \Return \(\texttt{Neither}\)
    \EndIf
\EndIf

\end{algorithmic}
\end{algorithm}

For any $k\ge 2$, the $k$-subdivision of a directed graph $D$, denoted by $D^{(k)}$, is a directed graph obtained by replacing each directed edge in $D$ by a directed path with $k$ edges oriented all in the original direction. \Cref{thm:cycles NOT TS/TAS} has the following interesting corollary.

\begin{corollary}\label{cor:subdvs cyc not TAS}
Let $k\ge 2$, $\ell\ge 1$, and let $\hat{C}_{2\ell}$ be the oriented cycle of length $2\ell$ with alternating orientation. Then $\hat{C}_{2\ell}^{(k)}$ is not LTAS.
\end{corollary}

\begin{proof}
The length of $\hat{C}_{2\ell}^{(k)}$ is $2\ell k$ and the number of flips from $C_\ell$ is $\ell k$. If one of $\ell$ and $k$ is even, then $2\ell k$ is divisible by 4 and $\ell k$ is even; if both $\ell$ and $k$ are odd, then $2\ell k$ is not divisible by 4 and $\ell k$ is odd. In either case, by \Cref{thm:cycles NOT TS/TAS}, $\hat{C}_{2\ell}^{(k)}$ is not LTAS.
\end{proof}

By the same proof, \Cref{thm:wedges_cycle} part $(iii)$ implies the following result.
\begin{corollary}
Let $k\ge 3$, $\ell\ge 1$, and let $\hat{C}_{2\ell}$ be the oriented cycle of length $2\ell$ with alternating orientation. Then $\hat{C}_{2\ell}^{(k)}$ is neither LTS nor LTAS.
\end{corollary}

The only unsolved case is when $k=2$.
\begin{problem}
Let $\ell\ge 1$, and let $\hat{C}_{2\ell}$ be the oriented cycle of length $2\ell$ with alternating direction. Is $\hat{C}_{2\ell}^{(2)}$ LTS or not?
\end{problem}

\section{Typical paths}\label{sec:typical}

As noted in Theorem \ref{thm:wedges}, for a path $P$, when it is nonzero $\cC(P_3,P)$ determines the behavior of its homomorphism counts into tournaments that are very close to quasirandom. Lemma \ref{lem:decomposition} gives that $\cC(P_3,P)$ is determined by summing the signs of all its copies of $P_3$, where $\rightarrow \rightarrow, \leftarrow \leftarrow$ are positive and $\rightarrow \leftarrow, \leftarrow \rightarrow$ are negative.

Considering now a randomly oriented path $P_k$ on $k$ vertices, where each edge is oriented right or left independently with equal probability, $\cC(P_3,P_k)$ is then distributed as a random $\pm 1$ walk taking $k-2$ steps. Thus, while we exhibited paths which are neither locally TS or locally TAS, we can say there are few of them.
\begin{proof}[Proof of \Cref{prop:localwalk}]
    By Theorem \ref{thm:wedges}, for an oriented path $P$ to be neither locally TS nor TAS it is necessary that $\mathcal{C}(P_3,P)=0$. Since this number is a random walk if $P$ has random orientation, the probability that $\mathcal{C}(P_3,P)=0$ can be bounded by the probability a random walk starting at $0$ taking $n$ steps ends back at $0$. This probability is $0$ if $n$ is odd and ${n \choose n/2}2^{-n}$ if $n$ is even. The $\frac 1 2$ statements then follow by symmetry of the random walk (and hence the sign of $\mathcal{C}(P_3,P)$ when it is nonzero).
\end{proof}

It is natural to ask if a similar phenomenon to Proposition \ref{prop:localwalk} holds for being globally TS/TAS and not just locally. Theorem \ref{thm:rarity}, which we now prove, shows that the TS part of this statement is false globally. However, as stated in Conjecture \ref{conj:ubiquity}, it still may hold for TAS.

\begin{proof}[Proof of Theorem \ref{thm:rarity}]

We show the result by exhibiting an explicit weighted tournament into which a randomly oriented path has low weighted homomorphism count (compared to quasirandom) with high probability. Our construction is very simple, with a more nuanced analysis. Consider the two-vertex graph $K$ consisting of a directed edge $(a,b)$ of weight $w(a,b) = 1$ and two self loops $(a,a),(b,b)$ of weight $w(a,a) = w(b,b) = 1/2$ on the endpoints. We think of our random path $P_n$ as arriving edge by edge and track the weighted homomorphism count of the currently revealed path into $K$. Explicitly, let $P_n$ be a randomly oriented path on $n$ edges with vertex set $0,1,\dots,n$ and for $i = 0,1,\dots n$ let $P_i$ be the subpath induced by the vertices $0,\dots,i$. Define $f_i$ to be the total weight of homomorphisms $\phi: P_i \hookrightarrow K$ where $\phi(i)=b$ if $(i-1,i) \in P_n$ and $\phi(i) = a$ if $(i,i-1) \in P_n$, recalling that the weight of such $\phi$ is $\prod_{k=1}^i w(\phi(k-1),\phi(k))$. Roughly speaking, $f_i$ counts those homomorphisms where the last vertex is ``consistent'' with the orientation of the last edge. Let $g_i$ be the total weight of the remaining ``inconsistent'' homomorphisms $P_i \hookrightarrow K$. Then by definition the total weighted homomorphism count $P_i \hookrightarrow K$ is $f_i + g_i$ and our goal is to understand the random variable $f_n + g_n$. 

Say a vertex in a path is balanced if it has one edge in and one edge out, and imbalanced otherwise. Let $I_{i}, i \in [n-1]$ be the independent indicator of the event that vertex $i$ is balanced in $P_n$, and let $I_0 = 1$. From the above description it can be seen that $\set{(f_i,g_i)}_{i=1}^n$ is a Markov process with the following evolution:
\begin{align*}
    &f_0 = g_0 = 1\\
    &f_i = I_{i-1}\bigp{\frac{f_{i-1}}{2} + g_{i-1}} + (1-I_{i-1})\bigp{\frac{g_{i-1}}{2} + f_{i-1}}\\
    &g_{i} = I_{i-1}\frac{g_{i-1}}{2} + (1-I_{i-1})\frac{f_{i-1}}{2}.
\end{align*}
$\set{f_i+g_i}_{i=1}^n$ is in fact a martingale, but we will not use this. The quasirandom benchmark we compare $f_n + g_n$ to is $\frac{2^{n+1}}{2^{n}} = 2$, the homomorphism count of $P_n$ into the quasirandom tournament. Thus our goal is to show that with high probability $f_n + g_n < 2$. We will in fact show that with high probability $f_n + g_n < 2e^{-n/1000}$, giving that all but a vanishing proportion of oriented paths are in a sense exponentially far from being tournament Sidorenko.

We will do this by borrowing from the theory of stochastic linear recurrences. Note that
\begin{align*}
f_n + g_n &= I_{n-1} \left( \frac{1}{2} f_{n-1} + \frac{3}{2} g_{n-1} \right) + (1 - I_{n-1}) \left( \frac{1}{2} g_{n-1} + \frac{3}{2} f_{n-1} \right) \\
&= \frac{1}{2} (f_{n-1} + g_{n-1})(I_{n-1} + 1 - I_{n-1}) + I_{n-1} g_{n-1} + (1 - I_{n-1}) f_{n-1} \\
&= (f_{n-1} + g_{n-1}) + \left( I_{n-1} - \frac{1}{2} \right) g_{n-1} + \left( \frac{1}{2} - I_{n-1} \right) f_{n-1}\\
&= (f_{n-1} + g_{n-1}) + \left( I_{n-1} - \frac{1}{2} \right) (g_{n-1}-f_{n-1}).
\end{align*}
On the other hand, subtracting the recurrence relations for \(f\) and \(g\), we find
\[
g_n - f_n = I_{n-1} \left( -\frac{1}{2} g_{n-1} - \frac{1}{2} f_{n-1} \right) + (1 - I_{n-1}) \left( -\frac{1}{2} f_{n-1} - \frac{1}{2} g_{n-1} \right) = -\frac 1 2 (f_{n-1}+g_{n-1}).
\]
Thus
\begin{align*}
    f_n + g_n &= (f_{n-1} + g_{n-1}) + \left( I_{n-1} - \frac{1}{2} \right) (g_{n-1}-f_{n-1}) \\
    &= (f_{n-1} + g_{n-1}) - \frac 1 2(I_{n-1} - 1/2)f_{n-2} + g_{n-2}) \\
    &= (f_{n-1} + g_{n-1}) \pm \frac{1}{4} (f_{n-2} + g_{n-2}).
\end{align*}
where $\pm$ is used to denote a uniformly random sign taken independently across $n \in \NN$.

Define $x_n = \frac 1 2 (f_n + g_n)$ so that
$$
x_n = x_{n-1} \pm \frac{1}{8} x_{n-2}, \quad x_1 = x_0 = 1
$$
For the stronger ``exponentially far from TS'' statement it suffices to show that with high probability $x_n < e^{-n/1000}$. Note that by definition $f_n + g_n$ is always positive, so $x_n$ is as well. A result of Furstenberg \cite{FurstLLN}, the law of large numbers for random matrix products, gives that the almost sure limit
$$
\lambda \coloneqq \lim_{n \to \infty} \frac{\ln x_n}{n}
$$
exists. In this context $\lambda$ is called the Lyapunov exponent. Of course this implies convergence in probability, and so
$$
\pr{x_n \ge e^{-n/1000}} = \pr{\frac{\ln x_n}{n} \ge -10^{-3}} \le \pr{\abs{\frac{\ln x_n}{n} - \lambda} \ge -\bigp{\lambda + 10^{-3}}}
$$
vanishes as long as $\lambda < - 10^{-3}$. To show this we use the method of \cite{fibseqsire} to bound the Lyapunov exponent, which in general is not an easy task. There is plenty of literature numerically exploring the dependence of this exponent on the recurrence parameters \cites{visfib,embreefib}, and in particular for our parameters \cite{embreefib} estimates $\lambda \approx -0.0083$ but does not provide any rigorous bounds. First we transform the process by considering the ratio $r_n \coloneqq x_{n+1}/x_n$, so that
$$
r_n = 1 \pm \frac{1/8}{r_{n-1}}, \quad r_0 = 1.
$$
Note that by definition $r_n$ is always positive with probability one (and thus in fact always greater than $1/8$). This process converges to an invariant distribution $P$, which from the above dynamics satisfies
$$
P(r) = \int_\RR \frac 1 2 \bigp{\delta\bigp{r-1+\frac{1/8}{q}} + \delta\bigp{r-1-\frac{1/8}{q}}} dq.
$$
By definition and the law of large numbers we then have that almost surely
$$
\lambda = \lim_{n \to \infty} \frac{1}{n} \sum_{j=1}^n \ln r_j = \EE_P \ln r = \int_\RR \ln(r) P(r) dr.
$$
Plugging in the invariance formula for $P(r)$ and expanding the logarithms gives
$$
\lambda = \frac 1 2 \int_\RR \bigp{\ln \bigp{1+\frac{1/8}{r}} + \ln \bigp{1-\frac{1/8}{r}}} P(r) dr = - \sum_{k=1}^\infty \frac{\bigp{1/8}^{2k}}{2k} \int_\RR r^{-2k} P(r) dr.
$$
Note now that, from the invariance property of $P$ and the fact that it is supported in $\RR_+$, the smallest $r$ such that $P(r) > 0$ satisfies $r_* = 1-\frac{1/8}{r_*}$ and the largest $r$ in the support satisfies $r^* = 1+ \frac{1/8}{r_*}$. Solving the system of equations gives that $P$ is supported within the interval $\bigb{\frac{1+\sqrt{2}}{2\sqrt{2}}, \frac{3\sqrt{2}-1}{2\sqrt{2}}}$. Thus throwing away all negative terms except the first gives
$$
\lambda < -\frac{(1/8)^2}{2\bigp{\frac{3}{\sqrt{2}}}^2} \int_\RR P(r) dr = -\frac{1}{576} < -10^{-3}
$$
as desired.
\end{proof}

\section{Toward a conjecture on trees}\label{sec:treecon}

This section proves Theorems \ref{thm:isopairtas} and \ref{thm:caterpillar} as well as discusses some related results. For this we will need a stronger notion of TAS, first defined in \cite{fox2024variationssidorenkosconjecturetournaments}, which we repeat here.
\begin{definition}[Strongly TAS]\label{def:strongtas}
    Given a digraph \(D\) and an independent set \(I\subset V(D)\), we say that the pair \((I,D)\) is \emph{strongly tournament anti-Sidorenko} (or $D$ is strongly TAS at $I$) if for any \(n\)-vertex tournament \(T\) and any embedding
\[
\phi\colon I\hookrightarrow V(T),
\]
the number of labeled copies \(\varphi\) of \(D\) in \(T\) with \(\varphi|_{I}=\phi\) is at most
\[
2^{-e(D)}\,n^{\,v(D)-\lvert I\rvert}.
\] 
\end{definition}
Note that if $I' \subseteq I$ then $(I',D)$ is strongly TAS if $(I,D)$ is strongly TAS. Since $(\emptyset,D)$ strongly TAS is just the same as $D$ being TAS, strongly TAS implies TAS as one would hope. Note that the analagous notion of strongly tournament Sidorenko, with ``at most'' replaced by ``at least'', is not interesting. In particular, if $(I,D)$ is strongly TS in this sense then every vertex in $I$ is isolated in $D$. Indeed if $(I,D)$ is any pair for which there is a vertex $v \in I$ that is not isolated in D, then letting $T$ be a transitive tournament and embedding $v$ as the source or sink we find an embedding $I \hookrightarrow T$ which does not extend to any copy of $D$.

We are now ready to give the key proposition motivating Definition \ref{def:isopair} of an isomorphic pair. Let $N(D,T)$ be the number of labeled copies of the digraph $D$ in the tournament $T$, and $N(D,T|v \hookrightarrow t)$ for $v \in D,t \in T$ be the number of labeled copies where $v$ is mapped to $t$.
\begin{proposition}[Isomorphic pair]\label{prop:amgmstrong}
    Let $H$ be any digraph and $w \in V(H)$ an arbitrarily designated vertex. Then letting $H'$ be a copy of $H$ and $w' \in V(H')$ any isomorphic image of $w$, the digraph $D$ defined by $V(D) = V(H) \cup V(H') \cup \set{v}$ and $E(D) = E(H) \cup E(H') \cup \set{(w,v),(v,w')}$ is such that for every tournament $T$ and vertex $t \in T$
    $$
    N(D,T|v \hookrightarrow t) \le N(H,T)^2/4
    $$
    Consequently if $H$ is TAS then $(\set{v},D)$ is strongly TAS.
\end{proposition}
\begin{proof}
By the AM-GM inequality
    \begin{align*}
        N(D,T|v \hookrightarrow t) &= N(H,T|w \hookrightarrow N^-(t))N(H',T|w' \hookrightarrow N^+(t)) \\
        &\le \bigp{\frac{N(H,T|w \hookrightarrow N^-(t)) + N(H',T|w' \hookrightarrow N^+(t))}{2}}^2\\
        &= N(H,T)^2/4
    \end{align*}
\end{proof}

Strongly TAS is primarily useful because two strongly TAS digraphs can be glued to form another.
\begin{lemma}[Strongly TAS gluing \cite{fox2024variationssidorenkosconjecturetournaments}*{Lemma 3.10}]\label{lem:gluing}
    If both $(I_1,D_1)$ and $(I_2,D_2)$ are strongly tournament anti-Sidorenko pairs with $v(D_2) \geq |I_1|$, then the digraph $D$ formed by identifying $|I_1|$ vertices of $D_2$ with $I_1 \subset V(D_1)$ satisfies that $(I_2,D)$ is strongly tournament anti-Sidorenko.
\end{lemma}

Combined with the above proposition this gives Theorem \ref{thm:isopairtas}, which we now recall.

\isopair*
\begin{proof}
    Suppose one is given an undirected graph $G$ containing an isomorphic pair $H_1,H_2$. Then by Proposition \ref{prop:amgmstrong} if there exists a TAS orientation of $H_1$ then orienting $H_2$ the same way and $\set{v,w}, \set{v,\phi(w)}$ as $(w,v),(v,\phi(w))$ gives a strongly TAS orientation of $H_1 \cup H_2 \cup \set{v}$ at $v$ (note that $v$ is not adjacent to anything else in $H_1,H_2$). Then if $G[V(G) \setminus (H_1 \cup H_2)]$ has a TAS orientation, by Lemma \ref{lem:gluing} gluing $H_1 \cup H_2 \cup \set{v}$ onto $G[V(G) \setminus (H_1 \cup H_2)]$ at $v$ gives a TAS orientation of all of $G$.
\end{proof}

The class of trees with no isomorphic pair is simpler, for instance no two leaves can share a parent as the two vertices of degree 1 in $K_{1,2}$ form an isomorphic pair adjacent to the vertex of degree 2. A small example of a tree with no isomorphic pair is the $1-2-3$ tree, depicted below with a TAS orientation (following from Theorem \ref{thm:caterpillar}).
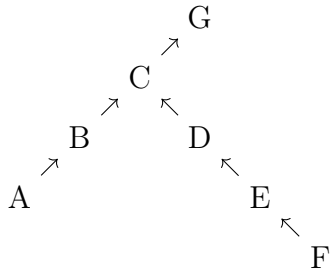
\begin{center}
\begin{tikzpicture}[scale = 0.8]
    \node (A) at (0,0) {A};
    \node (B) at (1,1) {B};
    \node (C) at (2,2) {C};
    \node (G) at (3,3) {G};
    \node (D) at (3,1) {D};
    \node (E) at (4,0) {E};
    \node (F) at (5,-1) {F};
    
    \draw[->] (A) -- (B);
    \draw[->] (B) -- (C);
    \draw[->] (C) -- (G);
    \draw[->] (D) -- (C);
    \draw[->] (E) -- (D);
    \draw[->] (F) -- (E);
\end{tikzpicture}
\end{center}

\subsection{Entropy method}
Our only proof that the above orientation of $1-2-3$ is TAS uses entropy, and is inspired by the proof of the TAS property of directed paths by Dong and Ślusarczyk (appendix of \cite{SSZ20}). The entropy method more generally gives Theorem \ref{thm:caterpillar}, which we now prove.

Given a discrete random variable $X$ with support $\Omega$, its entropy is 
$$
H(X)=-\sum_{x\in \Omega}\pr{X=x}\log(\pr{X=x}).
$$

Given another discrete random variable $Y$ with support $\Omega'$, the conditional entropy of $Y$ given $X$ is 
\begin{align*}
H(Y|X)&=\sum_{x\in\Omega}\pr{X=x}H(Y|X=x)\\
&=\sum_{x\in\Omega}\pr{X=x}\bigp{-\sum_{y\in \Omega'}\pr{Y=y|X=x}\log(\pr{Y=y|X=x}}.
\end{align*}

We make use of the following properties of entropy.

The uniform bound:
\begin{equation}
H(X)\le \log (|\Omega|).
\end{equation}

The chain rule:
\begin{equation}
H(X,Y)=H(X)+H(Y|X).
\end{equation}

See \cite{Z23}*{Section 5.5} for proofs of these properties as well as usages of entropy in the undirected setting of Sidorenko's conjecture. We can now construct a TAS orientation for every caterpillar.

\begin{proof}[Proof of Theorem \ref{thm:caterpillar}]
Let $X_0-X_1-\dots-X_{k+1}$ be the longest path in the caterpillar. For each $1\le i\le k$, let $s_i=d(X_i)-2$, where $d(X_i)$ is the number of neighbors of $X_i$. If $s_i>0$, then we denote the additional neighbors of $X_i$ as $Y_{i,1},\dots, Y_{i,s_i}$.

We oriented the caterpillar inductively as follows. We first let $X_0\rightarrow X_1$. For each $1\le i\le k$, if $X_{i-1}\rightarrow X_i$ ($X_{i-1}\leftarrow X_i$), 
then we let $X_i\rightarrow Y_{i,j}$ ($X_i\leftarrow Y_{i,j}$) if $j$ is odd, and $X_i\leftarrow Y_{i,j}$ ($X_i\rightarrow Y_{i,j}$) if $j$ is even. Further, if $d(X_i)$ is even, then we let $X_i\rightarrow X_{i+1}$ ($X_i\leftarrow X_{i+1}$); if $d(X_i)$ is odd, then we let $X_i\leftarrow X_{i+1}$($X_i\rightarrow X_{i+1}$). See the figure below for an example.
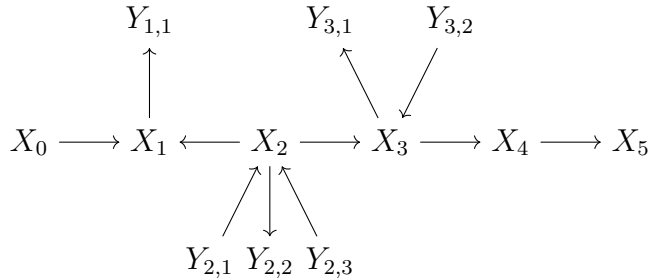
\begin{center}
\begin{tikzpicture}[scale = 0.8]
    \node (X0) at (0,0) {$X_0$};
    \node (X1) at (2,0) {$X_1$};
    \node (X2) at (4,0) {$X_2$};
    \node (X3) at (6,0) {$X_3$};
    \node (X4) at (8,0) {$X_4$};
    \node (X5) at (10,0) {$X_5$};

    \node (Y11) at (2,2) {$Y_{1,1}$};
    \node (Y21) at (3,-2) {$Y_{2,1}$};
    \node (Y22) at (4,-2) {$Y_{2,2}$};
    \node (Y23) at (5,-2) {$Y_{2,3}$};
    \node (Y31) at (5,2) {$Y_{3,1}$};
    \node (Y32) at (7,2) {$Y_{3,2}$};
    
    \draw[->] (X0) -- (X1);
    \draw[<-] (X1) -- (X2);
    \draw[->] (X2) -- (X3);
    \draw[->] (X3) -- (X4);
    \draw[->] (X4) -- (X5);
    \draw[->] (X1) -- (Y11);
    \draw[<-] (X2) -- (Y21);
    \draw[->] (X2) -- (Y22);
    \draw[<-] (X2) -- (Y23);
    \draw[->] (X3) -- (Y31);
    \draw[<-] (X3) -- (Y32);
\end{tikzpicture}
\end{center}

Consider a random homomorphism $f$ from the oriented caterpillar to the given tournament $T$ (with self-loops weighted 1/2) taken uniformly from all possible homomorphisms. For simplicity, we slightly abuse the notation and use $X_i$ and $Y_{i,j}$ to represent the random variables $f(X_i)$ and $f(Y_{i,j})$.

By the chain rule, we have
\begin{align*}
&H(X_0,\dots,X_{k+1},Y_{1,1},\dots,Y_{1,s_1},\dots,Y_{k,1},\dots,Y_{k,s_k})\\
=&H(X_0,X_1)+\sum_{i=1}^k\left(\sum_{j=1}^{s_i}H(Y_{i,j}|X_i)+H(X_{i+1}|X_i)\right),
\end{align*}
and similarly,
\begin{align*}
&H(X_0,\dots,X_{k+1},Y_{1,1},\dots,Y_{1,s_1},\dots,Y_{k,1},\dots,Y_{k,s_k})\\
=&H(X_k,X_{k+1})+\sum_{i=1}^k\left(\sum_{j=1}^{s_{k+1-i}}H(Y_{k+1-i,j}|X_{k+1-i})+H(X_{k-i}|X_{k+1-i})\right).
\end{align*}
Hence, taking an average, 
\begin{align*}
&H(X_0,\dots,X_{k+1},Y_{1,1},\dots,Y_{1,s_1},\dots,Y_{k,1},\dots,Y_{k,s_k})\\
=&\frac{H(X_1,X_2)+H(X_k,X_{k+1})}{2}+\sum_{i=1}^k\left(\sum_{j=1}^{s_i}H(Y_{i,j}|X_i)+\frac{H(X_{i-1}|X_{i})+H(X_{i+1}|X_{i})}{2}\right).   
\end{align*}

Note that for any three vertices $A,B,C$ in the caterpillar, if $A\rightarrow B$ and $A\leftarrow C$, then by the uniform bound and the AM-GM inequality
\begin{align*}
H(B|A)+H(C|A)&=\sum_{x\in V(T)}\pr{A=x}\bigp{H(B|A=x)+H(C|A=x)}\\
&\le \sum_{x\in V(T)}\pr{A=x}\bigp{\log(d^+_T(x))+\log(d^-_T(x))}\\
&\le \sum_{x\in V(T)}\pr{A=x}\log\bigp{\frac{n^2}{4}}\\
&=\log\bigp{\frac{n^2}{4}}.
\end{align*}

Note that in the term $\sum_{j=1}^{s_i}H(Y_{i,j}|X_i)+\frac{H(X_{i-1}|X_{i})+H(X_{i+1}|X_{i})}{2}$ the number of in-neighbors of $X_i$ equals the number of out-neighbors. Thus
$$
\sum_{j=1}^{s_i}H(Y_{i,j}|X_i)+\frac{H(X_{i-1}|X_{i})+H(X_{i+1}|X_{i})}{2}\le \frac{s_i+1}{2}\log(\frac{n^2}{4}).
$$
Also note that by the uniform bound
$$
H(X_1,X_2)+H(X_k,X_{k+1})\le 2\log(\frac{n^2}{2}).
$$
Therefore, we have
\begin{align*}
&H(X_0,\dots,X_{k+1},Y_{1,1},\dots,Y_{1,s_1},\dots,Y_{k,1},\dots,Y_{k,s_k})\\
&\le \log(\frac{n^2}{2})+\sum_{i=1}^k\frac{s_i+1}{2}\log(\frac{n^2}{4})\\
&=\log\left(\frac{n^{\sum_{i=1}^k s_i+k+2}}{2^{\sum_{i=1}^k s_i+k+1}}\right).
\end{align*}
Thus, by the uniform bound, we know that the number of homomorphisms is at most $\frac{n^{\sum_{i=1}^k s_i+k+2}}{2^{\sum_{i=1}^k s_i+k+1}}$, which implies that the oriented caterpillar has the TAS property.
\end{proof}




\subsection{Sparse graphs with no TAS orientation}

Another approach towards Conjecture \ref{conj:originaltree} is to ask: for which undirected graphs $G$ is it true that $G$ has a TAS orientation? In \cite{fox2024variationssidorenkosconjecturetournaments} it is shown that if $G$ has $k$ vertices and more than $k\log k$ edges, then no orientation of $G$ is TAS. We may thus define an undirected graph to be \textit{non-TAS} if every orientation is non-TAS.

\begin{proposition}\label{prop:sparse} For large enough $k$ there exists an undirected graph $G$ with $k$ vertices and $(\frac{1}{2} + o(1)) k \log k$ edges that is non-TAS.
\end{proposition}
\begin{proof}
    Let $G$ be an arbitrary graph on $k$ vertices with the property that $G$ can be partitioned into $m = (1+o(1)) \sqrt{k \log k}$ independent sets where between each pair there is exactly one edge. Thus, $G$ has the correct number $\binom{m}{2} \sim (\frac{1}{2} + o(1)) k \log k$ of edges. Such $G$ can be constructed easily by putting down $m$ independent sets first and building an edge between each pair of parts arbitrarily, and with a tiny bit more work we may guarantee all degrees in $G$ are $(1+o(1)) \log k$.

    We claim that every orientation of $G$ is not TAS. Indeed, observe by construction that every orientation $\vec{G}$ of $G$ has a homomorphism to some tournament $T$ on $m$ vertices. In particular, the homomorphism density of $\vec{G}$ into this $T$ is at least $(1/m)^k > 2^{-(\frac{1}{2}+o(1)) k \log k} > 2^{-e(\vec{G})}$ for appropriate choice of lower-order terms.    
\end{proof}

\bibliographystyle{amsplain}
\bibliography{refs.bib}


\appendix

\section{Proof of sign flip Proposition \ref{prop:signflip}}\label{sec:signflipapp}
Recall from the proof of \Cref{lem:lambda} that for a skew symmetric matrix $B$ with eigenvalues $\lambda_1 i,\dots,\lambda_n i$ we defined $X_{2t} \coloneqq \abs{1^\top B^{2t} 1}$ and proved that there exists $c_1,\dots,c_n \in \RR$ with $\sum_{i=1}^n c_i^2 = 1$ such that
\begin{equation}\label{eq:olda1}
    X_{2t} = \sum_i c_i^2 \lambda_i^{2t}.
\end{equation}

\begin{lemma}\label{lemma:estimating X} For any $s\ge t\ge 0$,
\begin{itemize}
    \item[(i)] $X_{2t+1}=0$;
    \item [(ii)] $X_{4t+2}=-1^\top B^{4t+2}1$ and $X_{4t}=1^\top B^{4t}1$;
    \item [(iii)] $X_{2s}\le X_{2t}(\frac{n}{2})^{2(s-t)}$;
    \item [(iv)] $X_{2s}^2\le X_{2(s-t)}X_{2(s+t)}$.
\end{itemize}
\end{lemma}

\begin{proof}
Equality $(i)$ follows from the fact that $B^{2i-1}$ is skew symmetric. Equality $(ii)$ is because $B^{4i-2}$ is negative semidefinite and $B^{4i}$ is positive semidefinite. Note that every entry of $B$ has absolute value at most 1/2, so the spectral radius is at most $n/2$, i.e. $|\lambda_t|\le n/2$ for every $1\le t\le n$. Then using \Cref{eq:olda1}
$$
X_{2s}=\sum_{k=1}^nc_k^2\lambda_k^{2s}\le (n/2)^{2(s-t)}\sum_{k=1}^nc_k^2\lambda_k^{2t}=X_{2t} (n/2)^{2(s-t)}.
$$
because $c_k \in \RR$. This proves inequality $(iii)$.

By the Cauchy-Schwarz inequality,
$$
X^2_{2s}=\bigp{\sum_{k=1}^nc_k^2\lambda_k^{2s}}^2\le \bigp{\sum_{k=1}^nc_k^2\lambda_k^{2(s-t)}}
\bigp{\sum_{k=1}^nc_k^2\lambda_k^{2(s+t)}}=X_{2(s-t)}X_{2(s+t)}.
$$
again by \Cref{eq:olda1}, which proves inequality $(iv)$.
\end{proof}

\begin{proof}[Proof of Proposition \ref{prop:signflip}]
The path $\rightarrow\leftarrow\leftarrow$ is known to be impartial (both TS and TAS)~\cite{ZHAO19}.
The path $\rightarrow\leftarrow\leftarrow\leftarrow$: the weighted homomorphism count of such paths in the given tournament is
\begin{align*}
1^\top  A (A^\top )^3 1&=1^\top  \bigp{\frac 12 J + B} \bigp{\frac 12 J - B}^3 1\\
&= \frac{1}{16} 1^\top  J^4 1 + \frac{1}{8} 1^\top  (BJ^3 - JBJ^2 - J^2BJ - J^3 B)1 \\
    &+ \frac{1}{4} 1^\top  (-BBJJ - BJBJ-BJJB + JBBJ + JBJB + JJBB)1 \\
    &+ \frac{1}{2} 1^\top  (BBBJ + BBJB + BJBB - JBBB) 1 \\
    &- 1^\top  B^4 1.
\end{align*}
Hence, by Lemma~\ref{lemma:estimating X},
$$
1^\top  A (A^\top )^3 1= \frac{n^5}{16}-\frac{n^2}{4}X_2-X_4\le \frac{n^5}{16}.
$$
The path $\rightarrow\rightarrow\leftarrow\leftarrow$: 
the weighted homomorphism count of such paths in the given tournament is
\begin{align*}
1^\top  A^2 (A^\top )^2 1&=1^\top  \bigp{\frac 12 J + B}^2 \bigp{\frac 12 J - B}^2 1\\
&= \frac{1}{16} 1^\top  J^4 1 + \frac{1}{8} 1^\top  (BJ^3 + JBJ^2 - J^2BJ - J^3 B)1 \\
    &+ \frac{1}{4} 1^\top  (BBJJ - BJBJ-BJJB - JBBJ - JBJB + JJBB)1 \\
    &+ \frac{1}{2} 1^\top  (-BBBJ - BBJB + BJBB +JBBB) 1 \\
    &+ 1^\top  B^4 1.
\end{align*}
Hence, by Lemma~\ref{lemma:estimating X},
$$
1^\top  A (A^\top )^3 1= \frac{n^5}{16}-\frac{n^2}{4}X_2+X_4\le \frac{n^5}{16}-\frac{n^2}{4}X_2+\frac{n^2}{4}X_2= \frac{n^5}{16}.
$$
The path $\rightarrow\leftarrow\leftarrow\leftarrow\leftarrow$: Similar as above, by definition and Lemma~\ref{lemma:estimating X}
\begin{align*}
1^\top  A (A^\top )^4 1&=\frac{n^6}{32}-2\frac{n^3}{8}X_2-X_2^2\le \frac{n^6}{32}.
\end{align*}
The path $\rightarrow\rightarrow\leftarrow\leftarrow\leftarrow$: By definition and Lemma~\ref{lemma:estimating X}
\begin{align*}
1^\top  A^2 (A^\top )^3 1&=\frac{n^6}{32}-2\frac{n^3}{8}X_2+\frac{1}{2}X_2^2
\le\frac{n^6}{32}-\frac{n^3}{4}X_2+\frac{n^2}{8}X_0X_2
\le \frac{n^6}{32}.
\end{align*}
The path $\rightarrow\leftarrow\leftarrow\leftarrow\leftarrow\leftarrow$: By definition and Lemma~\ref{lemma:estimating X}
\begin{align*}
1^\top  A (A^\top )^5 1
&=\frac{n^7}{64}-3\frac{n^4}{16}X_2-\frac{n}{4}X_2^2+\frac{n^2}{4}X_4+X_6\\
&\le \frac{n^7}{64}-\frac{3n^4}{16}X_2+\frac{2n^4}{16}X_2
\le \frac{n^7}{64}.
\end{align*}
The path $\rightarrow\rightarrow\leftarrow\leftarrow\leftarrow\leftarrow$: By definition and Lemma~\ref{lemma:estimating X}
\begin{align*}
1^\top  A^2 (A^\top )^4 1
&=\frac{n^7}{64}-3\frac{n^4}{16}X_2+\frac{n}{4}X_2^2+\frac{n^2}{4}X_4-X_6\\
&\le \frac{n^7}{64}-\frac{3n^4}{16}X_2+\frac{2n^4}{16}X_2
\le \frac{n^7}{64}.
\end{align*}
The path $\rightarrow\rightarrow\rightarrow\leftarrow\leftarrow\leftarrow$: By definition and Lemma~\ref{lemma:estimating X}
\begin{align*}
1^\top  A^3 (A^\top )^3 1
&=\frac{n^7}{64}-3\frac{n^4}{16}X_2+3\frac{n}{4}X_2^2-\frac{n^2}{4}X_4+X_6\\
&\le \frac{n^7}{64}-\frac{3n^4}{16}X_2+\frac{3n^4}{16}X_2-\frac{n^2}{4}X_4+\frac{n^2}{4}X_4=\frac{n^7}{64}.
\end{align*}
The path $\rightarrow\leftarrow\leftarrow\leftarrow\leftarrow\leftarrow\leftarrow$: By definition and Lemma~\ref{lemma:estimating X}
\begin{align*}
1^\top  A (A^\top )^6 1
&=\frac{n^8}{128}-4\frac{n^5}{32}X_2+2\frac{n^3}{8}X_4+2\frac{1}{2}X_2X_4\\
&\le \frac{n^8}{128}-\frac{4n^5}{32}X_2+\frac{4n^5}{32}X_2
\le \frac{n^8}{128}.
\end{align*}
The path $\rightarrow\rightarrow\leftarrow\leftarrow\leftarrow\leftarrow\leftarrow$: By definition and Lemma~\ref{lemma:estimating X}
\begin{align*}
1^\top  A^2 (A^\top )^5 1
&=\frac{n^8}{128}-4\frac{n^5}{32}X_2+2\frac{n^2}{8}X^2_2+2\frac{n^3}{8}X_4-2\frac{1}{2}X_2X_4\\
&\le \frac{n^8}{128}-\frac{4n^5}{32}X_2+\frac{4n^5}{32}X_2
= \frac{n^8}{128}.
\end{align*}
The path $\rightarrow\rightarrow\rightarrow\leftarrow\leftarrow\leftarrow\leftarrow$: By definition and Lemma~\ref{lemma:estimating X}
\begin{align*}
1^\top  A^3 (A^\top )^4 1
&=\frac{n^8}{128}-4\frac{n^5}{32}X_2+4\frac{n^2}{8}X^2_2\\
&\le \frac{n^8}{128}-\frac{4n^5}{32}X_2+\frac{4n^5}{32}X_2
= \frac{n^8}{128}.
\end{align*}
\end{proof}
The same proof does not work for the length-8 path $\rightarrow\leftarrow\leftarrow\leftarrow\leftarrow\leftarrow\leftarrow\leftarrow$, because 
\begin{align*}
1^\top  A (A^\top )^7 1
&=\frac{n^9}{256}-5\frac{n^6}{64}X_2+2\frac{n^3}{16}X_2^2+3\frac{n^4}{16}X_4+\frac{1}{4}X_2^3+2\frac{n}{4}X_2X_4-\frac{n^2}{4}X_6-X_8,
\end{align*}
and Lemma~\ref{lemma:estimating X} does not seem to be sufficient to imply that this is at most $\frac{n^9}{256}$.

\end{document}